\documentclass{cjour}
\usepackage{cjourps}
\begin{document}

%%%%% To be entered at Academic Press: =====>>

% \journame{}
% \articlenumber{}
% \yearofpublication{}
% \volume{}
% \cccline{}
% \received{}
% \revised{}
% \accepted{}

\authorrunninghead{Lai and Roach}
\titlerunninghead{Construction of Bivariate Symmetric Wavelets}

% communication line, use: \commline{Communicated by...}
% \commline{ }

%\setcounter{page}{261} %% This command is optional.

%% <<== End of commands to be entered at Academic Press

%%  Authors, start here ==>>

%\draft % Optional, will cause a line at the bottom of each page
%% with the words `Draft' and the time and date that the article
%% was LaTeXed. Will also double space text.

\title{Construction of bivariate symmetric orthonormal wavelets
with short support}

%\subtitle{}

%\author{Ming-Jun Lai\thanks{Supported by the National Science Foundation
%under grant DMS-9870187}}
%\affil{Mathematics Department, University of Georgia}

%%%%%%%%%%%%
%% More than one author with separate affiliations, either:
%
%\author{Ming-Jun Lai$^\dagger$
%\thanks{Supported by the National Science Foundation
%under grant DMS-9870187} and David W. Roach$^\ddagger$}
%\affil{$^{\dagger}$University of Georgia, $^{\ddagger}$Murray State
%University}

% or

\author{Ming-Jun Lai\thanks{Supported by the National Science Foundation
under grant DMS-9870187}}
\affil{University of Georgia}
\and
\author{David W. Roach}
\affil{Murray State University}
%%%%%%%%%%%%

%% \thanks command:
%% Can use \thanks{} in title to have footnote number appear and
%% footnote at the bottom of the page. i.e.,
%% \title{This is the title\thanks{Supported by grant no....}}

%% In \authors or \affil, can use \thanks{} to have asterisk,
%%   dagger or double dagger appear
%%   and text appear at the bottom of the title page. i.e.,

%\authors{D. Adalsteinsson and J. A. Sethian\thanks{Supported in part by the
%Applied Mathematics Subprogram of the...}}

%%%%%%%%%%%%

\email{mjlai@math.uga.edu and david.roach@murraystate.edu}

%optional
%\dedication{Dedicated to...}

\abstract{
In this paper, we give a parameterization of the class of bivariate
symmetric orthonormal scaling functions with filter size $6\times 6$
using the standard dilation matrix $2I$.
In addition, we give two families of refinable
functions which are not orthonormal but have associated tight frames.
Finally, we show that the class of bivariate symmetric
scaling functions with filter size $8\times 8$ can not have two or
more vanishing moments.}

% text should be lower case, unless caps are necessary for meaning

\begin{article}

% \contents is optional, will make a list of all section heads
% that appear in the article
%\contents
\input epsf.tex
\def\psone#1#2{\centerline{
    \epsfxsize #2in \epsfbox{#1}
    }}
\def\pstwo#1#2#3#4#5#6{
   \centerline{
    #5\epsfxsize=#2in \epsfbox{#1}
    #6\epsfxsize=#4in  \epsfbox{#3}
   }}
\def\xd{3}
%optional
% for those that like to start with section zero:
%\zerosection{Introduction}
%%%%%%%%%%%%%%
\section{Introduction}
%%%%%%%%%%%%%%
\label{Section 1}
The most common wavelets used for image processing are the
tensor-product of univariate compactly supported orthonormal wavelets.
Of this class of wavelets, only the Haar wavelet is symmetric which
gives its associated filter the property of linear phase.  Since
Daubechies' work\cite{D}, numerous generalizations of wavelets have
been developed including  biorthogonal wavelets, multiwavelets,
and bivariate wavelets.
Since 1992, several examples of bivariate compactly supported
orthonormal and biorthogonal wavelets have been constructed. See
Cohen and Daubechies'93 \cite{CD} for nonseparable bidimensional wavelets,
J.~Kova\v{c}evi\'c and M.~Vetterli'92\cite{KV} for nonseparable filters and
wavelets based on a generalized dilation matrix, He and Lai'97\cite{HL97} for
the complete solution of bivariate compactly supported wavelets with filter
size up to $4\times 4$, Belogay and Wang'99\cite{BW} for a special construction
of bivariate nonseparable wavelets for any given regularity, and Ayache'99
\cite{A} for nonseparable dyadic compactly supported wavelets with arbitrary
regularity. See also Cohen and Schlenker'93\cite{CS},
Riemenschneider and Shen'97\cite{RS}, and He and Lai'98\cite{HL98} for
bivariate biorthogonal box spline wavelets.

It is well-known that in the univariate setting, there does not exist
symmetric compactly supported orthonormal wavelets except Haar
for dilation factor $2$. We are interested in the construction of
symmetric wavelets in the bivariate setting with
dilation matrix $2I$ which have compact support and vanishing moments.
We start with a scaling funtion $\phi$. Let
$$
\hat{\phi}(\omega_1, \omega_2)= \prod_{k=1}^\infty m(e^{\omega_1/2^k},
e^{\omega_2/2^k})
$$
be the Fourier transform of $\phi$, where
$$
m(x,y)= \sum_{j=0}^N\sum_{k=0}^N c_{jk} x^j y^k
$$
is a trigonometric polynomial satisfying $m(1,1)=1$.
In addition, the trigonometric polynomial $m(x,y)$
 satisfies the orthonormality condition
$$|m(x,y)|^2+|m(-x,y)|^2+|m(x,-y)|^2+|m(-x,-y)|^2=1.$$
Let $\psi_i(x,y)$ be the corresponding wavelet satisfying
$$
\hat{\psi}_i(\omega_1,\omega_2)= m_i(e^{\omega_1/2},e^{\omega_2/2})
\hat{\phi}(\omega_1/2, \omega_2/2), i=1,2,3,
$$
where the $m_i$ are trigonometric polynomials such that the following
matrix
$$
\left[\begin{array}{cccc}
m(x,y) & m(-x,y) & m(x,-y) & m(-x,-y) \cr
m_1(x,y) & m_1(-x,y) & m_1(x,-y) & m_1(-x,-y) \cr
m_2(x,y) & m_2(-x,y) & m_2(x,-y) & m_2(-x,-y) \cr
m_3(x,y) & m_3(-x,y) & m_3(x,-y) & m_3(-x,-y) \cr
\end{array}\right]
$$
is unitary. Moreover, we are interested in  symmetric
scaling functions $\phi$ with a certain number of vanishing moments
 in the sense that their
associated trigonometric polynomial $m(x,y)$ satisfies
$$
m(1/x,1/y)= x^{-N}y^{-N} m(x,y)
$$
as well as
$$
\left.{\partial^k \over \partial x^k} m(x,y)\right|_{x=-1}=
\left.{\partial^k \over \partial y^k}m(x,y)\right|_{y=-1}=0, \hspace{.25in}
0\le k\le M-1.
$$
The symmetry condition provides $\phi$ with the property of linear phase
and the vanishing moment conditions provide $\phi$ with polynomial reproduction
up to degree $M-1$.
%a trigonometric polynomial $m(x,y)$ associated with a scaling function
%$\phi$ satisfies the symmetry property, then the trigonometric polynomials
%which define the wavelets are easily calculated as
If $m(x,y)$ satisfies the symmetric property, then the associated wavelets
can easily be found by using $\hat{\phi}$ and
\begin{eqnarray*}
m_1(x,y)&=&m(-x,y)\\
m_2(x,y)&=&x\cdot m(x,-y)\\
m_3(x,y)&=&x\cdot m(-x,-y).
\end{eqnarray*}

In summary, we are looking for trigonometric polynomials
$m(x,y)= \displaystyle \sum_{j=0}^N\sum_{k=0}^N c_{jk} x^j y^k$ which
satisfy the following properties:
\begin{enumerate}
\begin{enumerate}
\item Existence: $m(1,1)=1$ \label{(i)}.
\item Orthogonality: $\displaystyle |m(x,y)|^2+|m(-x,y)|^2+|m(x,-y)|^2+
|m(-x,-y)|^2=1$ \label{(ii)}.
\item Symmetry: $m(1/x,1/y)=x^{-N}y^{-N}m(x,y)$\label{(iii)}.
\item $M$ vanishing moments: $m(x,y)=(x+1)^M(y+1)^M\tilde{m}(x,y)$ where
$\tilde{m}(x,y)$ is another trigonometric polynomial.\label{(iv)}
\end{enumerate}
\end{enumerate}

In this paper, we construct a complete parameterization of
all trigonometric polynomials $m(x,y)$  which satisfy
the symmetry condition, the vanishing moment condition, and
the orthonormality condition for $N=5$ and $M=1$.  Within this
class, we identify a two-parameter family which contains the
trigonometric polynomials associated with scaling functions.
Outside of this two-parameter family, we show that the remaining
trigonometric functions are not associated with scaling functions
but instead determine families of tight frames.
Finally, we show that there are no trigonometric polynomials
for $N=7$ and $M=2$, and consequently no symmetric bivariate scaling
functions with two vanishing moments for the support size we are considering.

The paper is organized as follows.  Section 2 gives the parameterized
solution when $N=5$ and $M=1$.  The problem is broken down into
four cases which are dealt with in turn.
Section 3 discusses the orthonormality of the solutions from Section
2 and concludes with a  numerical experiment comparing Haar, D4, and
one solution from Section 2.  The last two sections show
that these trigonometric polynomials cannot have higher vanishing
moments(i.e $M\geq 2$) even for $N=7$.

%%%%%%%%%%%%%%
\section{The $6\times 6$ Case}
%%%%%%%%%%%%%%
\label{Section 2}
Our goal is to parameterize the coefficients of the trigonometric polynomials
which satisfy properties (i)-(iv).
We begin our investigation with trigonometric polynomials whose filter size
is $6\times 6$ with one vanishing moment, i.e. $N=5$ and $M=1$.  The case
$N=1$ is trivially the tensor product Haar function,
and the case $N=3$ has eight singleton solutions given by He and Lai'97
\cite{HL97}.

Let us express $m(x,y)$ in its polyphase form, i.e.,
\begin{eqnarray*}
m(x,y)&=&f_a(x^2,y^2)+xf_b(x^2,y^2)+ yf_c(x^2,y^2)+ xyf_d(x^2,y^2)\\
&=&\left[\begin{array}{c}1\\y\\y^2\\y^3\\y^4\\y^5\end{array}\right]^T
 \left[\matrix{
a_0 & b_0 & a_1 & b_1 & a_2 & b_2\cr
c_0 & d_0 & c_1 & d_1 & c_2 & d_2\cr
a_3 & b_3 & a_4 & b_4 & a_5 & b_5\cr
c_3 & d_3 & c_4 & d_4 & c_5 & d_5\cr
a_6 & b_6 & a_7 & b_7 & a_8 & b_8\cr
c_6 & d_6 & c_7 & d_7 & c_8 & d_8
}\right]\left[\begin{array}{c}1\\x\\x^2\\x^3\\x^4\\x^5\end{array}\right]
\end{eqnarray*}
where
\begin{eqnarray*}
f_\nu(x,y) &=&\nu_0+\nu_1x+\nu_2x^2+\nu_3y+
\nu_4xy+\nu_5x^2y+\nu_6y^2+\nu_7xy^2+\nu_8x^2y^2,
\end{eqnarray*}
for $\nu=a,b,c,d$.

The symmetry condition (iii) reduces the number of unknowns by half since
$m(x,y)$ becomes
$$
m(x,y)=\left[\begin{array}{c}1\\y\\y^2\\y^3\\y^4\\y^5\end{array}\right]^T
 \left[\matrix{
a_0 & b_0 & a_1 & b_1 & a_2 & b_2\cr
b_8 & a_8 & b_7 & a_7 & b_6 & a_6\cr
a_3 & b_3 & a_4 & b_4 & a_5 & b_5\cr
b_5 & a_5 & b_4 & a_4 & b_3 & a_3\cr
a_6 & b_6 & a_7 & b_7 & a_8 & b_8\cr
b_2 & a_2 & b_1 & a_1 & b_0 & a_0
}\right]\left[\begin{array}{c}1\\x\\x^2\\x^3\\x^4\\x^5\end{array}\right].
$$

For convenience, we denote $\displaystyle \sum_{\nu=a,b} \nu = a + b$.
Thus, by (i), we have
\begin{equation}
m(1,1)=2\sum^8_{i=0} \sum_{\nu=a,b} \nu_i = 2\sum^8_{i=0}(a_i + b_i) = 1.
\label{(2.1)}
\end{equation}

By (ii), we have the following 13 nonlinear equations

\begin{eqnarray}
&\displaystyle\sum_{\nu=a,b}& \nu_0 \nu_8 = 0 \label{(2.2)}\\
&\displaystyle\sum_{\nu=a,b}&  \nu_2 \nu_6 = 0 \label{(2.3)}\\
 &\displaystyle\sum_{\nu=a,b}& (\nu_1\nu_6+\nu_2\nu_7) = 0 \label{(2.4)}\\
 &\displaystyle\sum_{\nu=a,b}& (\nu_0\nu_7+\nu_1\nu_8) = 0 \label{(2.5)}\\
 &\displaystyle\sum_{\nu=a,b}& (\nu_2\nu_3+\nu_5\nu_6) = 0 \label{(2.6)}\\
 &\displaystyle\sum_{\nu=a,b}& (\nu_0\nu_5+\nu_3\nu_8) = 0 \label{(2.7)}\\
 &\displaystyle\sum_{\nu=a,b}& (\nu_0\nu_6+\nu_1\nu_7+\nu_2\nu_8) = 0
\label{(2.8)}\\
 &\displaystyle\sum_{\nu=a,b}& (\nu_0 \nu_2 + \nu_3\nu_5 + \nu_6\nu_8) = 0
\label{(2.9)}\\
 &\displaystyle\sum_{\nu=a,b}& (\nu_1\nu_3 + \nu_2\nu_4 + \nu_4\nu_6 +
\nu_5\nu_7) = 0 \label{(2.10)}\\
 &\displaystyle\sum_{\nu=a,b}& (\nu_0\nu_4+\nu_1 \nu_5 + \nu_3 \nu_7+
\nu_4\nu_8) = 0 \label{(2.11)}\\
 &\displaystyle\sum_{\nu=a,b}& (\nu_0\nu_3+\nu_1\nu_4 +\nu_2\nu_5+
\nu_3\nu_6+\nu_4\nu_7+\nu_5\nu_8)=0 \label{(2.12)}\\
 &\displaystyle\sum_{\nu=a,b}& (\nu_0\nu_1+\nu_1\nu_2+\nu_3\nu_4 +
\nu_4\nu_5+\nu_6\nu_7+\nu_7\nu_8) = 0 \label{(2.13)}\\
 &\displaystyle\sum^8_{i=0}&
\displaystyle\sum_{\nu=a,b} \nu^2_i = {1\over 8}. \label{(2.14)}
\end{eqnarray}

The first moment condition (iv) for $M=1$ yields the following
six linear equations:
\begin{eqnarray}
&a_0 + a_1 + a_2 = b_0 + b_1 + b_2 \label{(2.15)}\\
&a_3 + a_4 + a_5 = b_3 + b_4 + b_5 \label{(2.16)}\\
&a_6 + a_7 + a_8 = b_6 + b_7 + b_8 \label{(2.17)}\\
&a_0 + a_3 + a_6 = b_2 + b_5 + b_8 \label{(2.18)}\\
&a_1 + a_4 + a_7 = b_1 + b_4 + b_7 \label{(2.19)}\\
&a_2 + a_5 + a_8 = b_0 + b_3 + b_6. \label{(2.20)}
\end{eqnarray}

We need to find the $a_i$'s and $b_i$'s
which satisfy the equations (\ref{(2.1)})-(\ref{(2.20)})
simultaneously.  We proceed by specifying necessary conditions
derived from these equations.

\begin{lemma}\label{Lemma 2.2}
$\displaystyle \sum^8_{i=0} a_i \ =\
\displaystyle \sum^8_{i=0} b_i\  =\  {1\over 4}$.
\end{lemma}
\begin{proof} Adding (\ref{(2.15)})- (\ref{(2.17)}) together,
we have $\displaystyle \sum^8_{i=0} a_i = \sum^8_{i=0} b_i$.
By equation (\ref{(2.1)}), the result follows.
\end{proof}

Next, equations (\ref{(2.2)}), (\ref{(2.3)}), (\ref{(2.8)}),
(\ref{(2.9)}), and (\ref{(2.14)}) imply
\begin{equation}
\sum_{\nu=a,b}(\nu_4^2 + (\nu_1 + \nu_7)^2 + (\nu_3 + \nu_5)^2 + (\nu_0 +
\nu_2 + \nu_6 + \nu_8)^2) = {1\over 8}. \label{(2.21)}
\end{equation}
We use equations (\ref{(2.10)}) and (\ref{(2.11)}) to get
\begin{equation}
\sum_{\nu=a,b}(\nu_4(\nu_0 + \nu_2 + \nu_6 + \nu_8) + (\nu_1 + \nu_7)(\nu_3 +
\nu_5)) = 0.\label{(2.22)}
\end{equation}
From equations (\ref{(2.4)}), (\ref{(2.5)}), and (\ref{(2.13)}), we have
\begin{equation}
\sum_{\nu=a,b}((\nu_1+\nu_7)(\nu_0 + \nu_2 + \nu_6 + \nu_8) + \nu_4(\nu_3 +
\nu_5)) = 0.\label{(2.23)}
\end{equation}
Also, we have, using equations (\ref{(2.6)}), (\ref{(2.7)}), and (\ref{(2.12)}),
\begin{equation}
\sum_{\nu=a,b}( \nu_4(\nu_1 + \nu_7)+ (\nu_0 + \nu_2 + \nu_6 + \nu_8)(\nu_3 +
\nu_5)) = 0.\label{(2.24)}
\end{equation}

\begin{lemma}\label{Lemma 2.3}
\begin{eqnarray*}
a_0 + a_2 + a_4 + a_6 + a_8 &=& {1\over 8} + {1\over
4\sqrt{2}} \cos \alpha,\\
b_0 + b_2 + b_4 + b_6 + b_8 &=& {1\over 8} + {1\over
4\sqrt{2}} \sin \alpha.
\end{eqnarray*}
\end{lemma}
\begin{proof}
By equation (\ref{(2.21)}) and (\ref{(2.22)}), we have
$$
\sum_{\nu=a,b}\left((\nu_4+(\nu_0 + \nu_2 + \nu_6 + \nu_8))^2 + ((\nu_1 + \nu_7) +
(\nu_3 + \nu_5))^2\right) = {1\over 8}.
$$
By Lemma \ref{Lemma 2.2} and letting $a^* = a_0 + a_2 + a_4
+ a_6 + a_8$ and $b^* = b_0 + b_2 + b_4 + b_6 +
b_8$, we have
$$
(a^*)^2 + ({1\over 4} - a^*)^2 + (b^*)^2 +
({1\over 4} - b^*)^2 = {1\over 8}
$$
or
$$
(a^* - {1\over 8})^2 + (b^* - {1\over 8})^2 =
{1\over 32}.
$$
Thus, we are able to conclude the proof.
\end{proof}

Let $\hat{a} = a_1+ a_4 + a_7$.  Observe that equation (\ref{(2.19)})
implies $\hat{a}= b_1 + b_4 + b_7$.
It follows from (\ref{(2.21)}) and (\ref{(2.24)}) that
$$
\sum_{\nu=a,b} \left((\nu_4+(\nu_1 + \nu_7))^2 +
((\nu_0 + \nu_2 + \nu_6 + \nu_8)+
(\nu_3 + \nu_5))^2\right) = {1\over 8}.
$$
Thus, $2(\hat{a})^2 + 2({1\over 4} - \hat{a})^2 =
{1\over 8}$.  That is, we have
$\hat{a} = 0\ \  {\rm or }\  \ {1\over 4}.
$
Similarly, it follows from (\ref{(2.21)}) and (\ref{(2.23)}) that
$$
\sum_{\nu=a,b}\left((\nu_3 + \nu_4 + \nu_5)^2 +
((\nu_0 + \nu_2 + \nu_6 + \nu_8) +(\nu_1 + \nu_7))^2\right) = {1\over 8}.
$$
Let $\tilde{a}=a_3+a_4+a_5=b_3+b_4+b_5$ where the second equality comes from
equation (\ref{(2.16)}).  Thus,
$2(\tilde{a})^2 + 2({1\over 4} - \tilde{a})^2 =
{1\over 8}$.  That is,
$
\tilde{a} = 0\  \ {\rm or } \ \ \tilde{a} =
{1\over 4}.
$
Therefore, we have the following four cases to consider:
\begin{itemize}
\item Case 1: \ \ $a_1+a_4+a_7=1/4$\ \  and\ \  $a_3+a_4+a_5=1/4$
\item Case 2: \ \ $a_1+a_4+a_7=1/4$\ \  and\ \  $a_3+a_4+a_5= 0$
\item Case 3: \ \ $a_1+a_4+a_7= 0$\ \  \hskip .135in and\ \  $a_3+a_4+a_5=1/4$
\item Case 4: \ \ $a_1+a_4+a_7= 0$\ \  \hskip .135in and\ \  $a_3+a_4+a_5= 0$
\end{itemize}

Since Case 3 is a rotation of Case 2, we will only study Cases 1,2, and 4.
As will be discussed, the different cases are actually conditions
associated with the zeros of $m(x,y)$ along the axes.

%%%%%%%%%%%%
\subsection{Complete Solution of Case 1}
%%%%%%%%%%%%
\label{section 2.1}
\noindent Let us first consider Case 1 where both
$a_1+a_4+ a_7 = {1\over 4}$ and $a_3 + a_4 + a_5 = {1\over 4}$.
By Lemmas \ref{Lemma 2.2} and \ref{Lemma 2.3}, we have
$$
a_1 + a_3 + a_5 + a_7\ =\  {1\over 8} - {1\over 4\sqrt{2}}\cos \alpha, \ \ \
b_1 + b_3 + b_5 + b_7\ =\  {1\over 8} - {1\over 4\sqrt{2}}\sin \alpha.
$$
However, $(a_1 +a_4 +a_7)+(a_3 +a_4+a_5) = {1\over 2}$.
It then follows that
$$
2a_4 = {1\over 2} - {1\over 8} + {1\over
4\sqrt{2}}\cos \alpha \ \ \ {\rm or }\ \ \ a_4 =
{3\over 16} + {1\over 8\sqrt{2}} \cos \alpha.
$$
Similarly, $b_4 = {3\over 16} + {1\over 8\sqrt{2}} \sin \alpha$.
Let us now summarize the above discussion combined with Lemma \ref{Lemma 2.3}
in the following lemma.

\begin{lemma}\label{Lemma 3.1}
\begin{eqnarray*}
a_4  = {3\over 16} + {1\over 8\sqrt{2}} \cos \alpha,
&&\hspace{.85in}b_4  = {3\over 16} + {1\over 8\sqrt{2}} \sin\alpha,\\
a_1 + a_7 = {1\over 16} - {1\over 8\sqrt{2}} \cos \alpha,
&&\hspace{.55in} b_1 + b_7 = {1\over 16} - {1 \over 8\sqrt{2}} \sin \alpha,\\
a_3 + a_5 = {1\over 16} - {1\over 8\sqrt{2}} \cos \alpha,
&&\hspace{.55in} b_3 + b_5 = {1\over 16} - {1\over 8\sqrt{2}} \sin\alpha,\\
a_0 + a_2 + a_6 + a_8  =   -{1\over 16} + {1\over 8\sqrt{2}} \cos \alpha,
&& b_0 + b_2 + b_6 + b_8 = -{1\over 16} + {1\over 8\sqrt{2}} \sin \alpha.
\end{eqnarray*}
\end{lemma}
We are now ready to solve for the 18 unknowns.
The complete parameterization of Case 1 is given in Theorem \ref{Theorem 3.2}.

\begin{theorem}\label{Theorem 3.2}
For any $\beta, \gamma \in [0, 2\pi]$, let

\begin{eqnarray*}
\alpha=2(\beta-\gamma)+\frac{\pi}{4},\hspace{.2in} p={1\over 16} - {1\over 8\sqrt{2}}\cos \alpha, \ &{\rm and}&\
q={1\over 16} - {1\over 8\sqrt{2}} \sin \alpha.
\end{eqnarray*}
If
\begin{eqnarray*}
a_0 &=& (-p(1 + \cos(\beta - \gamma)) - q \sin
(\beta - \gamma) - \sqrt{p^2 + q^2} (\cos \beta +
\cos \gamma))/4\\
a_2 &=&(-p(1 - \cos(\beta - \gamma)) + q \sin (\beta
- \gamma) - \sqrt{p^2 + q^2}(\cos \beta - \cos
\gamma))/4\\
a_6 &=&(-p(1 - \cos(\beta - \gamma)) + q\sin(\beta
- \gamma) + \sqrt{p^2 + q^2}(\cos \beta - \cos
\gamma))/4\\
a_8 &=&(-p(1 + \cos(\beta - \gamma)) - q(\sin(\beta
- \gamma) + \sqrt{p^2 + q^2} (\cos \beta + \cos
\gamma))/4\\ \\
b_0 &=&(-q(1 + \cos(\beta - \gamma)) + p\sin(\beta
- \gamma) - \sqrt{p^2 + q^2}(\sin \beta + \sin
\gamma))/4\\
b_2 &=&(-q(1 - \cos(\beta - \gamma)) - p\sin(\beta -
\gamma) - \sqrt{p^2 + q^2}(\sin \beta - \sin
\gamma))/4\\
b_6 &=&(-q(1 - \cos (\beta - \gamma)) - p\sin(\beta
- \gamma) + \sqrt{p^2 + q^2}(\sin \beta - \sin
\gamma))/4\\
b_8 &=&(-q(1 + \cos (\beta - \gamma)) + p\sin(\beta
- \gamma) + \sqrt{p^2 + q^2}(\sin \beta + \sin
\gamma))/4\\ \\
  a_1 &=& {p\over 2} + {1\over 2}\sqrt{p^2 + q^2} \cos \beta, \hspace{.2in}
  b_1 \ =\  {q\over 2} + {1\over 2}\sqrt{p^2 + q^2} \sin\beta  \\
  a_3 &=& {p\over 2} + {1\over 2}\sqrt{p^2 + q^2} \cos\gamma, \hspace{.2in}
  b_3\ =\ {q\over 2} + {1\over 2}\sqrt{p^2 + q^2} \sin\gamma  \\
  a_5 &=& {p\over 2} - {1\over 2}\sqrt{p^2 + q^2}\cos\gamma, \hspace{.2in}
  b_3\ =\ {q\over 2} - {1\over 2}\sqrt{p^2 + q^2}\sin\gamma\\
  a_7 &=& {p\over 2} - {1\over 2}\sqrt{p^2 + q^2}\cos\beta, \hspace{.2in}
  b_7\ =\ {q\over 2} - {1\over 2}\sqrt{p^2 + q^2 }\sin\beta \\ \\
  a_4 &=& {1\over 4} - p, \hspace{.2in}   b_4 \ =\ {1\over 4} -q,
\end{eqnarray*}
\noindent then $m(x,y)$ with these coefficients $a_i$ and $b_i$ satisfies
the properties (i), (ii), (iii), and (iv). On the other hand, if
$m(x,y)$ satisfies the properties (i), (ii), (iii), (iv),
and the conditions of Case 1, then the coefficients
$a_i$ and $b_i$ can be expressed in the above format.
\end{theorem}
\begin{proof}
The proof consists of deriving necessary conditions from
various combinations of the nonlinear equations which eventually
leads to a complete parameterization that satisfies each
nonlinear equation individually.

Assume the conditions of Case 1, i.e. $a_1+a_4+a_7=1/4$
and $a_3+a_4+a_5=1/4$.  Using equations (\ref{(2.9)}) and (\ref{(2.14)}),
we have
$$
\sum_{\nu=a,b}((\nu_0 + \nu_2)^2 + (\nu_3 + \nu_5)^2 + (\nu_6 + \nu_8)^2
+ \nu_1^2 + \nu_4^2 + \nu_7^2) = {1\over 8}.
$$
Using (\ref{(2.13)}), the above equation becomes
$$
\sum_{\nu=a,b}((\nu_0 + \nu_1 + \nu_2)^2 + (\nu_3 + \nu_4 + \nu_5)^2 +
(\nu_6 + \nu_7 + \nu_8)^2) = {1\over 8}.
$$
Since $a_3 + a_4 + a_5 = b_3 + b_4 + b_5  = {1\over 4}$, we have
\begin{equation}
a_0 + a_1 + a_2 = 0,\ \  a_6 + a_7 + a_8 = 0, \ \ b_0 + b_1 + b_2 = 0, \ \
b_6 + b_7 + b_8 = 0. \label{(3.1)}
\end{equation}
Similarly, by equations (\ref{(2.8)}) and (\ref{(2.14)}), we have
$$
\sum_{\nu=a,b}((\nu_0 + \nu_6)^2 + (\nu_1 + \nu_7)^2 + (\nu_2 + \nu_8)^2
+ \nu^2_3 + \nu^2_4 + \nu^2_5) = {1\over 8}.
$$
Using (\ref{(2.12)}), the above equation becomes
$$
\sum_{\nu=a,b}((\nu_0 + \nu_3 + \nu_6)^2 + (\nu_1 + \nu_4 + \nu_7)^2 +
(\nu_2 + \nu_5 + \nu_8)^2) = {1\over 8}.
$$
Since $a_1 + a_4+ a_7 = b_1 + b_4 + b_7 = {1\over 4}$,  we have
\begin{equation}
a_0 + a_3 + a_6 = 0, \ \ a_2 + a_5 + a_8 = 0,
\ \ b_0 + b_3 + b_6 = 0, \ \ b_2 + b_5 + b_8 = 0. \label{(3.2)}
\end{equation}
We now use (\ref{(3.1)}), (\ref{(3.2)}), and the equations
in Lemma \ref{Lemma 3.1} to simplify the equations
(\ref{(2.2)})-(\ref{(2.14)}).  By (\ref{(3.1)}),  equation (\ref{(2.13)})
becomes
\begin{eqnarray*}
0 &=& \sum_{\nu=a,b} (-\nu_1^2 +\nu_4(\nu_3+\nu_5)-\nu_7^2)\\
&=&\sum_{\nu=a,b} -(\nu_1+\nu_7)^2+\nu_4({1\over 4}-\nu_4)+\nu_1\nu_7 \\
&=&\sum_{\nu=a,b}(-({1\over 4}- \nu_4)^2+\nu_4({1\over 4}-\nu_4)+\nu_1\nu_7)\\
&=& \sum_{\nu=a,b} \nu_1 \nu_7.
\end{eqnarray*}
since $\sum_{\nu=a,b} (-({1\over 4}- \nu_4)^2+\nu_4({1\over 4}-\nu_4))=0$ by
Lemma \ref{Lemma 3.1}.  A similar situation exists for equation
(\ref{(2.12)}), thus equations (\ref{(2.12)}) and (\ref{(2.13)}) simplify to
\begin{equation}
a_3 a_5 + b_3 b_5 = 0, \ \ a_1 a_7 + b_1 b_7 = 0.\label{(3.3)}
\end{equation}

We are now able to solve for $a_1, a_3, a_5, a_7, b_1, b_3, b_5,$ and
$b_7$. For simplicity, we denote $a_1+a_7=p$ and $b_1+b_7=q$. By
(\ref{(3.3)})  we have
$$
a^2_1 - pa_1+b^2_1 - b_1 q = 0 \ \ \ \hbox{or}\ \ \
(a_1-{p\over 2})^2+(b_1-{q\over 2})^2={1\over 4}(p^2  + q^2).
$$
Thus, we have
\begin{eqnarray*}
a_1 &=& {p\over 2} + {1\over 2}\sqrt{p^2+q^2} \cos \beta, \ \ \
b_1 \ =\  {q\over 2} + {1\over 2}\sqrt{p^2+q^2} \sin \beta, \\
a_7 &=& {p\over 2} - {1\over 2}\sqrt{p^2+q^2} \cos \beta, \ \ \
b_7 \ =\  {q\over 2} - {1\over 2}\sqrt{p^2+q^2} \sin \beta.
\end{eqnarray*}
Similarly, we have
\begin{eqnarray*}
a_3 &=& {p\over 2} + {1\over 2}\sqrt{p^2+q^2} \cos \gamma, \ \ \
b_3 \ =\ {q\over 2} + {1\over 2}\sqrt{p^2+q^2} \sin \gamma, \\
a_5 &=& {p\over 2} - {1\over 2}\sqrt{p^2+q^2} \cos \gamma, \ \ \
b_5 \ =\  {q\over 2} - {1\over 2}\sqrt{p^2+q^2} \sin \gamma.
\end{eqnarray*}
Next we simplify the equations (\ref{(2.2)})-(\ref{(2.11)}).  Note
that the addition of equation (\ref{(2.10)}) and (\ref{(2.11)}) is
\begin{equation}
\sum_{\nu=a,b}(\nu_4(\nu_0 + \nu_2 + \nu_6 + \nu_8) +
(\nu_1 + \nu_7)(\nu_3 +\nu_5)) = 0.\label{(3.4)}
\end{equation}
After substitution of the equations in Lemma \ref{Lemma 3.1},
the left hand side of equation (\ref{(3.4)}) becomes
\begin{eqnarray*}
&&2({1\over 4}-p)(-p) + 2(-p)(-p) + 2({1\over 4}-q)(-q) + 2(-q)(-q)\\
&=& 4[p^2 - {1\over 8} p + ({1\over 16})^2 + q^2 -
{1\over 8} q + ({1\over 16})^2] - 8({1\over 16})^2\\
&=& 4[(p - {1\over 16})^2 + (q - {1\over 16})^2] - {1\over 32}\\
&=& 4[{1\over 128}\cos^2\alpha+{1\over 128}\sin^2\alpha]- {1\over 32} = 0.
\end{eqnarray*}
That is, the equations (\ref{(2.10)}) and (\ref{(2.11)}) are linearly
dependent and consequently only one needs to be considered.  We will
deal with equation (\ref{(2.10)}) later.  Furthermore, for
equations (\ref{(2.4)})-(\ref{(2.9)}), we have the following
\begin{splitmath}
\sum_{\nu=a,b} ( \nu_1 \nu_6+\nu_2 \nu_7)
= \sum_{\nu=a,b} \nu_1 \nu_6 + \nu_1 \nu_7 + \nu_2 \nu_6+ \nu_2 \nu_7\\
= \sum_{\nu=a,b} (\nu_1 + \nu_2)(\nu_6 + \nu_7)
= \sum_{\nu=a,b}(-\nu_0)(-\nu_8) = \sum_{\nu=a,b} \nu_0 \nu_8,
\end{splitmath}
\begin{splitmath}
\sum_{\nu=a,b} ( \nu_0 \nu_7+\nu_1 \nu_8 )
= \sum_{\nu=a,b} \nu_0 \nu_7+ \nu_0 \nu_8+\nu_1 \nu_7+\nu_1\nu_8\\
= \sum_{\nu=a,b}(\nu_0 + \nu_1)(\nu_7 + \nu_8)
= \sum_{\nu=a,b}(-\nu_2)(-\nu_6) = \sum_{\nu=a,b} \nu_2 \nu_6,
\end{splitmath}
\begin{splitmath}
\sum_{\nu=a,b} (\nu_2 \nu_3+\nu_5 \nu_6)
= \sum_{\nu=a,b}  \nu_2\nu_3 + \nu_2\nu_6 + \nu_3\nu_5+ \nu_5\nu_6\\
=\sum_{\nu=a,b} (\nu_2 + \nu_5)(\nu_3 + \nu_6)
= \sum_{\nu=a,b}(-\nu_8)(-\nu_0) = \sum_{\nu=a,b} \nu_0 \nu_8,
\end{splitmath}
\begin{splitmath}
\sum_{\nu=a,b}(\nu_0 \nu_5 + \nu_3 \nu_8)
= \sum_{\nu=a,b} \nu_0 \nu_5 + \nu_0 \nu_8+ \nu_3\nu_5+ \nu_3 \nu_8\\
= \sum_{\nu=a,b}(\nu_0 + \nu_3)(\nu_5 + \nu_8)
= \sum_{\nu=a,b}(-\nu_6)(-\nu_2) = \sum_{\nu=a,b} \nu_2\nu_6,
\end{splitmath}
\begin{splitmath}
\sum_{\nu=a,b}(\nu_0 \nu_6 +\nu_1 \nu_7 + \nu_2 \nu_8)
= \sum_{\nu=a,b} \nu_0 \nu_6 + \nu_2 \nu_8
= \sum_{\nu=a,b} \nu_0 \nu_6 + \nu_0 \nu_8+ \nu_2 \nu_6 + \nu_2 \nu_8\\
= \sum_{\nu=a,b}(\nu_0 + \nu_2)(\nu_6 + \nu_8)
= \sum_{\nu=a,b} (-\nu_1)(-\nu_7) = \sum_{\nu=a,b} \nu_1\nu_7,
\end{splitmath}
\begin{splitmath}
\sum_{\nu=a,b} (\nu_0 \nu_2 + \nu_3 \nu_5 + \nu_6 \nu_8)
= \sum_{\nu=a,b} \nu_0 \nu_2  + \nu_0 \nu_8 + \nu_2 \nu_6+ \nu_6 \nu_8\\
= \sum_{\nu=a,b} (\nu_0 + \nu_6)(\nu_2 + \nu_8)
= \sum_{\nu=a,b} (-\nu_3)(-\nu_5) =\sum_{\nu=a,b} \nu_3 \nu_5.
\end{splitmath}
Therefore, only $\sum_{\nu=a,b} \nu_0 \nu_8 = 0$ and $\sum_{\nu=a,b} \nu_2
\nu_6 = 0$ remain to be solved in addition to (\ref{(2.10)}) since
equation (\ref{(2.14)}) follows from (\ref{(2.1)})-(\ref{(2.13)}).

In order to solve equations (\ref{(2.2)}) and (\ref{(2.3)}), note that
(\ref{(3.1)}) and (\ref{(3.2)}) imply $a_0 = a_5 + a_8 - a_1$ and
$b_0 = b_5 + b_8 -b_1$.  Thus, (\ref{(2.2)}) becomes
\begin{equation}
a^2_8 + a_8(a_5 - a_1) + b^2_8 + b_8(b_5 - b_1) =0\label{(3.5)}
\end{equation}
while (\ref{(2.3)}) becomes
$$
(a_5 + a_8)(a_7 + a_8) + (b_5 + b_8)(b_7 + b_8) = 0
$$
or
\begin{equation}
a^2_8 + a_8(a_5 + a_7) + b^2_8 + b_8(b_5 + b_7) =
-a_5 a_7 - b_5 b_7. \label{(3.6)}
\end{equation}
The subtraction of (\ref{(3.5)}) from (\ref{(3.6)}) yields
$$
a_8(a_1 + a_7) + b_8(b_1 + b_7) = -a_5 a_7 - b_5 b_7.
$$
So, we have
\begin{equation}
p a_8 + q b_8 = -a_5 a_7 - b_5 b_7.\label{(3.7)}
\end{equation}
The addition of (\ref{(3.5)}) and (\ref{(3.6)}) yields
$$
a^2_8 + (a_5 + {a_7 - a_1\over 2}) a_8 + b^2_8 +
(b_5 + {b_7 - b_1\over 2}) b_8 = -{1\over 2} (a_5
a_7 + b_5b_7)
$$
which is
\begin{equation}
(a_8 + {1\over 2}(a_5 + {a_7 - a_1 \over 2}))^2 +
(b_8 + {1\over 2}(b_5 + {b_7 - b_1\over 2}))^2 = R_1\label{R1}
\end{equation}
for some known value $R_1$.  In addition,  equation (\ref{(3.7)}) can be
rewritten as
\begin{equation}
p(a_8 + {1\over 2}(a_5 + {a_7 - a_1\over 2})) +
q(b_8 + {1\over 2}(b_5 + {b_7 - b_1\over 2})) = R_2\label{R2}
\end{equation}
for some known value $R_2$. Equations (\ref{R1}) and (\ref{R2})
can be solved simultaneously.  Thus, we
obtain the expression for the $a_i$'s and $b_i$'s given in Theorem \ref{Theorem 3.2}.

Finally, to satisfy (\ref{(2.10)}), we put these $a_i$ and $b_i$
into (\ref{(2.10)}) and simplify the equation yielding
$$
\cos(\beta-\gamma)={1\over \sqrt{2}}\cos(\alpha-\beta+\gamma)
+{1\over \sqrt{2}} \sin(\alpha-\beta+\gamma).
$$
Solving, we find that $\alpha=2(\beta-\gamma)+\pi/4$.

The above discussion shows that if $m(x,y)$ satisfies the
properties (i), (ii), (iii), (iv), and the conditions
of Case 1, then its coefficients $a_i$ and  $b_i$ can be expressed as
the two-parameter family given in the statement  of Theorem \ref{Theorem 3.2}.

On the other hand, to verify that $m(x,y)$ with the coefficients
$a_i$ and $b_i$ satisfies the properties (i), (ii), (iii), (iv),
and the conditions of Case 1, we just substitute the solutions back
into equations (\ref{(2.1)})-(\ref{(2.20)}). This completes the proof.
\end{proof}

%\noindent Figure \ref{Figure 2.1} shows one solution of this two parameter family which contains the tensor-product Haar scaling function whenever
%$\beta=\gamma$.

%\begin{center}
%\begin{figure}
%\psone{plot1a.ps}{5}
%\caption{A typical scaling function in the two-parameter family of Case 1.}
%\label{Figure 2.1}
%\end{figure}
%\end{center}

%%%%%%%%%%%%%
\subsection{Solution of Case 2}
%%%%%%%%%%%%%

In this section, we consider Case $2$ where
$a_1+a_4+a_7=1/4$\ \  and\ \  $a_3+a_4+a_5= 0$.  By
Lemmas \ref{Lemma 2.2} and \ref{Lemma 2.3},
we have $a_1 + a_3 + a_5 + a_7 = {1 \over
8} - {1 \over4\sqrt{2}} \cos \alpha$. So, we have
$$
a_4 = {1 \over 16} + { 1 \over 8 \sqrt{2}} \cos \alpha\ \ \hbox{and} \ \
b_4 = {1 \over 16} + {1 \over 8 \sqrt{2}}\sin \alpha.
$$
By (\ref{(2.9)}),(\ref{(2.13)}), and (\ref{(2.14)}), we have
$$
\sum_{\nu=a,b}\left( (\nu_0 + \nu_1 + \nu_2)^2 +
(\nu_3 + \nu_4 + \nu_5)^2 + (\nu_6 + \nu_7 +\nu_8)^2 \right)= {1 \over 8}.
$$
The assumptions of Case $2$ imply
$$
\sum_{\nu=a,b} \left((\nu_0 + \nu_1 + \nu_2)^2 +
(\nu_6 + \nu_7 + \nu_8)^2\right) = {1 \over 8},
$$
i.e.
$$
(a_0 + a_1 + a_2)^2 + (a_6 + a_7 + a_8)^2 + (b_0 + b_1 +
b_2)^2 + (b_6 + b_7 + b_8)^2 = {1 \over 8}.
$$
By (\ref{(2.15)}) and (\ref{(2.17)}), we have
$$
(a_0 + a_1 + a_2) ^2 + (a_6 + a_7 + a_8)^2= {1 \over 16}.
$$
By Lemma \ref{Lemma 2.2}, the above equation becomes
$$
( {1 \over 4} - (a_6 + a_7 + a_8))^2 + (a_6 + a_7 + a_8)^2 =
{1 \over 16}.
$$
It follows that $a_6 + a_7 + a_8 = 0$ or $a_6 + a_7 +
a_8 = {1 \over 4}$.
Thus, Case $2$  branches out into two subcases.

\begin{itemize}
\item Subcase 2a:
$\displaystyle
a_1 + a_4 + a_7 = {1 \over 4}, \ \ a_3 + a_4 + a_5 = 0, \ \
a_6 + a_7 + a_8 = 0, \ \ a_0 + a_1 + a_2 = {1\over 4}.
$
\item Subcase 2b: $\displaystyle
a_1 + a_4 + a_7 = {1 \over 4},  \ \ a_3 + a_4 +a_5 = 0, \ \
a_6 + a_7 + a_8 = {1 \over 4},  \ \ a_0 + a_1 +a_2 = 0.
$
\end{itemize}
We only consider subcase $2a$ here.  The subcase $2b$ can be
treated similarly and is left to the interested reader.
Theorem \ref{Theorem 4.1} gives the complete solution
of Subcase 2a.

\begin{theorem} \label{Theorem 4.1}

1)  For  any $\gamma \in [0,2\pi]$,\ \ \ $a_3=b_3=a_4=b_4=a_5=b_5=0$,
\begin{eqnarray*}
a_1 &=& {3\over 16} - {1 \over 8\sqrt{2}}\cos \gamma,\ \ \ \
a_7 = {1\over 16} + {1 \over 8\sqrt{2}}\cos \gamma,\\
a_0 &=& {1\over 32}\left(1 +\sqrt{2} \cos \gamma \pm
\sqrt{2 + \sqrt{2} (\cos \gamma + \sin \gamma)}\right),\\
a_2 &=& {1\over 32}\left(1 + \sqrt{2} \cos \gamma \mp
\sqrt{2 + \sqrt{2}(\cos \gamma + \sin \gamma)}\right),  \\
a_6 &=& {1\over 32}\left(-1 - \sqrt{2} \cos \gamma \mp
\sqrt{2 + \sqrt{2}(\cos \gamma + \sin \gamma)}\right),\\
a_8 &=& {1\over 32}\left(-1 -\sqrt{2} \cos \gamma \pm
\sqrt{2 + \sqrt{2} (\cos \gamma + \sin \gamma)}\right), \\
b_1 &=& {3\over 16} - {1 \over 8\sqrt{2}}\sin \gamma, \ \ \ \
b_7 = {1\over 16} +{1 \over 8\sqrt{2}}\sin \gamma,\\
b_0 &=& {1\over 32}\left(1 +\sqrt{2} \sin \gamma \mp
\sqrt{2 + \sqrt{2} (\cos \gamma + \sin \gamma)}\right),\\
b_2 &=& {1\over 32}\left(1 + \sqrt{2} \sin \gamma \pm
\sqrt{2 + \sqrt{2}(\cos \gamma + \sin \gamma)}\right),\\
b_6 &=& {1\over 32}\left(-1 - \sqrt{2} \sin \gamma \pm
\sqrt{2 + \sqrt{2}(\cos \gamma + \sin \gamma)}\right),\\
b_8 &=& {1\over 32}\left(-1 -\sqrt{2} \sin \gamma \mp
\sqrt{2 + \sqrt{2} (\cos \gamma + \sin \gamma)}\right).
\end{eqnarray*}
2) For any $\alpha \in [0,{\pi \over 4})\cup({\pi \over 4},2\pi]$, let
\begin{eqnarray*}
\gamma&=&\left\{\begin{array}{ccccc}
-{1\over 2}\alpha+{7\pi \over 8}& {\rm or} & {1\over 2}\alpha-{3\pi \over 8}&
{\rm if} & 0\leq \alpha <{ \pi \over 4}\\
 {1\over 2}\alpha+{5\pi \over 8}& {\rm or} & -{1\over 2}\alpha-{ \pi \over 8}&
 {\rm if} & { \pi \over 4}< \alpha \leq 2\pi
\end{array}\right.\\
p&=&-{1\over 16}-{1\over 8\sqrt{2}}\cos \alpha, \ \ \ \
q=-{1\over 16}-{1\over 8\sqrt{2}}\sin \alpha, \\
s&=&\sqrt{32(p^2+q^2)}, \ \ \ \
t=\sqrt{(p+{1\over 8})^2+(q+{1\over 8})^2},
\end{eqnarray*}
\begin{eqnarray*}
&&a_0=-{1\over 32}\left(-1+8p+s+{t\over p-q}\left[
(-s-8p+8q)\cos \gamma+s\sin\gamma\right]\right)\\
&&a_2={1\over 32}\left(1-8p+s+{t\over p-q}\left[
(-s+8p-8q)\cos \gamma+s\sin\gamma\right]\right)\\
&&a_6=-{1\over 32}\left(1+8p+s+{t\over p-q}\left[
(s+8p-8q)\cos \gamma-s\sin\gamma\right]\right)\\
&&a_8={1\over 32}\left(-1-8p+s+{t\over p-q}\left[
(s-8p+8q)\cos \gamma-s\sin\gamma\right]\right)\\
&&b_0={1\over 32}\left(1-8q+s+{t\over p-q}\left[
(s+8p-8q)\sin \gamma-s\cos\gamma\right]\right)\\
&&b_2=-{1\over 32}\left(-1+8q+s+{t\over p-q}\left[
(s-8p+8q)\sin \gamma-s\cos\gamma\right]\right)\\
&&b_6={1\over 32}\left(-1-8q+s+{t\over p-q}\left[
(-s-8p+8q)\sin \gamma+s\cos\gamma\right]\right)\\
&&b_8=-{1\over 32}\left(1+8q+s+{t\over p-q}\left[
(-s+8p-8q)\sin \gamma+s\cos\gamma\right]\right)
\end{eqnarray*}
\begin{eqnarray*}
&&a_1={1\over 16}(3+8p-8t\cos\gamma), \ \ \ \
b_1\ =\ {1\over 16}(3+8q-8t\sin\gamma)\\
&&a_7={1\over 16}(1+8p+8t\cos\gamma), \ \ \ \
b_7\ =\ {1\over 16}(1+8q+8t\sin\gamma)\\
&&a_3={1\over 16}(8p+s), \hspace{.75in}
b_3\ =\ {1\over 16}(8q-s),\\
&&a_5\ =\ {1\over 16}(8p-s), \hspace{.7in}
b_5\ =\ {1\over 16}(8q+s)\\
&&a_4=-p, \hspace{1.25in} b_4\ =\ -q.
\end{eqnarray*}
3) For  $\alpha={\pi \over 4}$,
$$
(c_{jk})_{j,k}={1\over 8}\left[\begin{array}{cccccc}
0 & \ \ 1 & 1 & 1 & \ \ 1 & 0\\
0 &\ \ 0 &0 &0 &\ \ 0& 0\\
0& -1 &1 &1& -1& 0\\
0& -1& 1& 1& -1& 0\\
0 &\ \ 0& 0& 0& \ \ 0 &0\\
0 &\ \ 1 & 1 & 1 & \ \ 1 & 0\\
\end{array}\right]
\ \ \ \hbox{or}\ \ \
{1\over 16}\left[\begin{array}{cccccc}
\ \ 1 & \ \ 1 & 2 & 2 & \ \ 1 & \ \ 1\\
-1 &\ \ 1 &0 &0 &\ \ 1& -1\\
\ \ 0& -2 &2 &2& -2& \ \ 0\\
\ \ 0& -2& 2& 2& -2& \ \ 0\\
-1 & \ \ 1& 0& 0& \ \ 1 &-1\\
\ \ 1 & \ \ 1 & 2 & 2 & \ \ 1 & \ \ 1\\
\end{array}\right].
$$
If $m(x,y)$ has the coefficients given by 1), 2), or 3), then
$m(x,y)$ satisfies the properties of (i), (ii), (iii), and (iv).  Conversely,
if $m(x,y)$ satisfies the properties of (i), (ii), (iii), (iv), and
the conditions of Subcase 2a, then the coefficients of $m(x,y)$ can be
expressed
as 1), 2), or 3).
\end{theorem}
\begin{proof}
Assume the conditions of Subcase 2a, i.e.
$$\displaystyle
a_1 + a_4 + a_7 = {1 \over 4}, \ \ a_3 + a_4 + a_5 = 0, \ \
a_6 + a_7 + a_8 = 0, \ \ a_0 + a_1 + a_2 = {1\over 4}.
$$
We immediately have from Lemma \ref{Lemma 2.3}
$$a_4=\frac{1}{16}+\frac{1}{8\sqrt{2}}\cos \alpha, \ \ \ \
b_4=\frac{1}{16}+\frac{1}{8\sqrt{2}}\sin \alpha.$$
By equations (\ref{(2.8)}), (\ref{(2.12)}) and (\ref{(2.14)}), we have
$$
\sum_{\nu=a,b} (\nu_0 + \nu_3 + \nu_6)^2 + (\nu_1 + \nu_4 + \nu_7)^2 + (\nu_2 + \nu_5 +
\nu_8)^2 = {1 \over 8}.
$$
It follows that
\begin{eqnarray}
a_0 + a_3 + a_6 &=& 0, \ \ \ \  b_0 + b_3 + b_6 = 0,\nonumber\\
a_2 + a_5 + a_8 &=& 0, \ \ \ \  b_2 + b_5 + b_8 = 0.\label{(4.1)}
\end{eqnarray}
We now simplify the 13 nonlinear equations (\ref{(2.2)}) -(\ref{(2.14)}).
With (\ref{(2.2)}), (\ref{(2.3)}), and (\ref{(4.1)}),
we may simplify (\ref{(2.9)}) as follows:
\begin{eqnarray}
0 &=& \sum_{\nu=a,b}(\nu_0 \nu_2 + \nu_3 \nu_5 + \nu_6 \nu_8)
\ =\  \sum_{\nu=a,b} (\nu_0 + \nu_6) (\nu_2 + \nu_8) + \nu_3 \nu_5 \nonumber\\
&=& \sum_{\nu=a,b} (-\nu_3) (-\nu_5) + \nu_3 \nu_5 \ =\  2 \sum_{\nu=a,b} \nu_3 \nu_5.
\label{(4.2)}
\end{eqnarray}

Recall $a_3 + a_4 + a_5 = 0$ and $ b_3 + b_4 + b_5 = 0$.  We can
now solve for $a_3, a_5, b_3, b_5$.  Indeed, (\ref{(4.2)}) can be rewritten
as $(a_4 +a_5) a_5 + (b_4 + b_5) b_5 = 0$.  After simplifying, we have
$$
(a_5 + {a_4 \over 2})^2 + (b_5 + {b_4 \over 2})^2 =
{1\over 4} (a_4^2 + b_4^2).
$$
Thus, we have
\begin{eqnarray*}
a_3 &=&  {p \over 2}  + {1 \over 2} \sqrt{p^2 + q^2}
\cos \beta, \ \
b_3 \ =\   {q \over 2}  + {1 \over 2} \sqrt{p^2 + q^2}
\sin \beta\\
a_5&=&  {p \over 2}  - {1 \over 2} \sqrt{p^2 + q^2}
\cos \beta, \ \
b_5 \ =\   {q \over 2}  - {1 \over 2} \sqrt{p^2 + q^2}
\sin \beta,
\end{eqnarray*}
where $p = -a_4$ and $q = -b_4$ (different $p$ and $q$ from section
\ref{section 2.1}).
We note that
\begin{equation}
p^2 + q^2 + {1 \over 8} ( p + q) = 0. \label{(4.3)}
\end{equation}
Similarly, we may simplify (\ref{(2.8)}) to be
\begin{eqnarray}
0&=& \sum_{\nu=a,b} (\nu_1 \nu_7 + \nu_0 \nu_6 + \nu_2 \nu_8)
\ =\ \sum_{\nu=a,b} \nu_1 \nu_7 + (\nu_0 + \nu_2) (\nu_6 + \nu_8)\nonumber\\
&=& \sum_{\nu=a,b} \nu_1 \nu_7 + ( {1 \over 4} - \nu_1) (-\nu_7)
\ =\ 2 \sum_{\nu=a,b} (\nu_1 \nu_7 - {1 \over 8} \nu_7).\label{(4.4)}
\end{eqnarray}
Recall $a_1 + a_7 = {1 \over 4} -a_4$ and $b_1 + b_7= {1 \over 4} - b_4.$
We may solve for $a_1$ and $b_1$ to get $a_1 = {1 \over 4} - a_4 - a_7$
and $b_1 = { 1 \over 4} - b_4 - b_7$.  Putting them into
(\ref{(4.4)}),  we have
$$
\left({1 \over 4} - a_4 - a_7 \right) a_7 - {a_7 \over 8} +
\left({1 \over 4} - b_4 - b_7 \right) b_7 - { b_7 \over 8} = 0
$$
or
$$
\left(a_7 -  {{1 \over 8} - a_4 \over 2}\right)^2 + \left(b_7 -
{{1 \over 8} - b_4 \over 2}\right)^2 = {\left(a_4 - {1 \over
8} \right)^2 + \left( b_4-{1
\over 8}\right)^2 \over 4}.
$$
It follows that
\begin{eqnarray*}
a_7  &=& {p \over 2} + {1 \over 16} + {1 \over 2}
\sqrt{\left( p+{1 \over 8} \right)^2 + \left(q+ {1 \over
8} \right)^2} \ \ \cos \gamma, \\
b_7 &=& {q \over 2} + {1 \over 16} + { 1 \over 2} \sqrt{
\left( p+ { 1 \over 8} \right)^2 + \left( q+ { 1 \over 8}
\right)^2} \ \ \sin \gamma,\\
a_1 &=& {p\over 2} + { 3 \over 16} - {1 \over 2} \sqrt{
\left( p+ {1 \over 8} \right)^2 + \left( q+ { 1 \over 8}
\right)^2} \ \ \cos \gamma,\\
b_1 &=& {q \over 2} + { 3 \over 16} - { 1 \over 2}
\sqrt{\left( p+ { 1 \over 8} \right)^2 + \left( q+ { 1
\over 8} \right)^2} \ \ \sin \gamma.\\
\end{eqnarray*}
We add the left-hand side of (\ref{(2.10)}) and (\ref{(2.11)})
together to get
\begin{eqnarray*}
\sum_{\nu=a,b}&&\hspace{-.2in}\bigg(\nu_0 + \nu_2\bigg) \nu_4 + \nu_1 \bigg(\nu_3 + \nu_5
\bigg) + \nu_4 \bigg(\nu_6 + \nu_8 \bigg) + \nu_7 \bigg(\nu_3 +
\nu_5\bigg)\\
&=&\sum_{\nu=a,b} \nu_4 \bigg( {1 \over 4} - \nu_1 \bigg) + \nu_1
\bigg(-\nu_4 \bigg) + \nu_4 \bigg( -\nu_7 \bigg) + \nu_7 \bigg(
-\nu_4 \bigg)\\
&=&\sum_{\nu=a,b} \nu_4 \left[ { 1 \over 4} - 2 \bigg( \nu_1 + \nu_7 \bigg)
\right]
\ =\  \sum_{\nu=a,b} \nu_4 \bigg(2\nu_4 - { 1 \over 4} \bigg)\\
&=&2 \bigg( 2a_4^2 - { 1 \over 4} a_4 \bigg) + 2 \bigg( 2 b_4^2
- { 1 \over 8} b_4 \bigg)\\
&=& 4 \left[ p^2 + q^2 + { 1 \over 8} \bigg( p + q \bigg)
\right] \ =\  0.
\end{eqnarray*}
That is, only one of (\ref{(2.10)}) and (\ref{(2.11)}) needs to be solved.
Thus, we deal with (\ref{(2.10)}) later.

Turning our attention to equation (\ref{(2.13)}).  We have
by (\ref{(4.2)}) and (\ref{(4.3)})
\begin{eqnarray*}
 \sum_{\nu=a,b}&& \hspace{-.2in}\nu_0 \nu_1 + \nu_1 \nu_2 + \nu_3 \nu_4 +
\nu_4 \nu_5 + \nu_6 \nu_7
+ \nu_7 \nu_8\\
&=& \sum_{\nu=a,b} \bigg(\nu_0 + \nu_2  \bigg) \nu_1 + \bigg(\nu_3 + \nu_5 \bigg)
\nu_4 + \nu_7 \bigg(\nu_6 + \nu_8 \bigg)\\
&=& \sum_{\nu=a,b} \bigg( { 1 \over 4} - \nu_1 \bigg)
\nu_1 - \nu_4^2 - \nu_7^2\\
&=&  \sum_{\nu=a,b} -2 \nu_1 \nu_7 -
\nu_1^2 - \nu_7^2 - \nu_4^2+{1\over 4}(\nu_1+\nu_7)\\
&=& \sum_{\nu=a,b} \bigg( - \bigg( \nu_1 + \nu_7 \bigg)^2 - \nu_4^2+
{1\over 4}(\nu_1+\nu_7) \bigg)\\
&=&  - \sum_{\nu=a,b} \bigg( { 1 \over 4} - \nu_4 \bigg)^2 + \nu_4^2-
{1\over 4}\bigg({1\over 4}-\nu_4\bigg)\ =\ 0.
\end{eqnarray*}
That is, (\ref{(2.13)}) holds for these $a_4, a_1, a_7, b_1,
b_7$.  Similarly, (\ref{(2.12)}) is satisfied in that
\begin{eqnarray*}
\sum_{\nu=a,b} &&\hspace{-.2in}\bigg( \nu_0 \nu_3 + \nu_1 \nu_4 +
\nu_2 \nu_5 + \nu_3 \nu_6 +\nu_4 \nu_7 + \nu_5 \nu_8 \bigg) \\
&=& \sum_{\nu=a,b} -\nu_3 \bigg(\nu_0 + \nu_6 \bigg) +
\nu_4 \bigg( \nu_1 + \nu_7\bigg) + \nu_5 \bigg( \nu_2 + \nu_8 \bigg)\\
&=& \sum_{\nu=a,b} \bigg(-\nu_3^2 + \nu_4
\bigg( { 1 \over 4} - \nu_4 \bigg) -\nu_5^2 \bigg) \\
&=& \sum_{\nu=a,b} - \bigg( \nu_3 + \nu_5 \bigg)^2 - a^2_4 + {\nu_4\over 4}\\
&=& -\sum_{\nu=a,b} \bigg( 2 \nu_4^2 - {\nu_4\over 4 }\bigg) = 0.
\end{eqnarray*}
We now show that equations (\ref{(2.3)}), (\ref{(2.9)}), (\ref{(4.1)})
and (\ref{(4.2)}) imply (\ref{(2.6)}).  The addition of
(\ref{(2.6)}) and (\ref{(2.9)}) yields
$$
\sum_{\nu=a,b} \nu_6 \bigg(\nu_5 + \nu_8 \bigg) + \nu_2 \bigg( \nu_0 + \nu_3
\bigg) + \nu_3 \nu_5 = \sum_{\nu=a,b} -2\nu_2 \nu_6  + \nu_3\nu_5 =0.
$$
Similarly, the equations (\ref{(2.2)}), (\ref{(2.9)}) and (\ref{(4.2)})
imply (\ref{(2.7)}).  Since (\ref{(2.8)}) holds for these $a_1, b_7, a_7, b_7$
under the assumptions of (\ref{(2.2)}) and (\ref{(2.3)}), we further simplify
(\ref{(2.4)}) by adding (\ref{(2.4)}) and (\ref{(2.8)}) together.  That is,
\begin{eqnarray*}
0 &=& \sum_{\nu=a,b} \nu_2 \nu_7 + \nu_1 \nu_6 +
\nu_1 \nu_7 + \nu_0 \nu_6 + \nu_2 \nu_8\nonumber\\
&=& \sum_{\nu=a,b} \nu_6 \bigg( \nu_0 + \nu_1 \bigg) +
\nu_2 \bigg( \nu_7 + \nu_8\bigg) + \nu_1 \nu_7 \nonumber\\
&=& \sum_{\nu=a,b} \nu_6 \bigg( { 1 \over 4} - \nu_2 \bigg) +
\nu_2 \bigg (-\nu_6\bigg) + \nu_1 \nu_7\nonumber\\
&=& \sum_{\nu=a,b} \bigg( { 1 \over 4 } \nu_6 + \nu_1 \nu_7 \bigg)
= \sum_{\nu=a,b}\left( {1\over 4} \nu_6+
{1\over 8} \nu_7\right).\label{(4.5)}
\end{eqnarray*}

Similarly, the sum of equations (\ref{(2.5)}) and (\ref{(2.8)}) is
equivalent to
\begin{equation}
\sum_{\nu=a,b} \bigg(\nu_1 \nu_7 + {1 \over 4} \nu_8 \bigg) =
\sum_{\nu=a,b} \left({1\over 8}\nu_7 +{1\over 4}\nu_8\right)=0. \label{(4.6)}
\end{equation}
However, equations (\ref{(4.5)}) and (\ref{(4.6)}) are equivalent by using
$a_6+a_7+a_8=0$ and $b_6+b_7+b_8=0$. Thus, we only need to consider one
of them.

In summary, we only need to solve the following
equations
$$
a_0 a_8 + b_0 b_8 = 0, \ \ \ \ \ a_2 a_6 + b_2 b_8
= 0,
$$
$$
\sum_{\nu=a,b}(\nu_7 + 2\nu_8) = 0, \ \ \ \sum_{\nu=a,b}(\nu_0 \nu_4 + \nu_4 \nu_8
+ \nu_1\nu_5 + \nu_3 \nu_7) = 0.
$$
Using the linear relationships, we have
\begin{eqnarray*}
a_0 &=& {1\over 4} - a_1 + a_5 + a_8, \ \ b_0 =
{1\over 4} - b_1 + b_5 + b_8\\
a_2 &=& -a_5 - a_8, \ \ \ b_2 = -b_5 - b_8,
\ \ \ a_6 = -a_7 - a_8, \ \ \ b_6 = -b_7 - b_8.
\end{eqnarray*}
Putting these linear equations in (\ref{(2.2)}) and (\ref{(2.3)}),
we get
\begin{eqnarray*}
a^2_8 + (a_5 - a_1 + {1\over 4}) a_8 + b^2_8 + (b_5
- b_1 + {1\over 4})b_8 &=& 0\\
a^2_8 + (a_5 + a_7) a_8 + a_5 a_7 + b^2_8 +(b_5 +
b_7)b_8 + b_5b_7 &=& 0.
\end{eqnarray*}
Subtracting the first one of the above two
equations from the second one, and using $a_1 + a_7
= p + {1\over 4}$, $b_1 + b_7 = q + {1\over 4}$, we
get
$$
pa_8 + qb_8 = -a_5 a_7 - b_5b_7.
$$
Using these linear relationships and (\ref{(2.11)}), we get
$$
\displaystyle pa_8 + qb_8 = {1\over 2} \sum_{\nu=a,b} (\nu_1 \nu_5 + \nu_3
\nu_7(\nu_5 - \nu_1 + {1\over 4})\nu_4).
$$
It follows that
$$
\displaystyle \sum_{\nu=a,b}(2\nu_5\nu_7 + \nu_1 \nu_5 + \nu_3 \nu_7 +
(\nu_5 - \nu_1 +{1\over 4})\nu_4) = 0
$$
which can be simplified to be
$$
\sum_{\nu=a,b}({1\over 4} \nu_5 + \nu_4^2) = 0
$$
That is, we have
\begin{eqnarray*}
0 &=& a^2_4 + b^2_4 + {1\over 4}(a_5 + b_5)\\
&=& {1\over 8}(a_4 + b_4) + {1\over 4}(p + q +
{1\over 2} \sqrt{p^2 + q^2}(\cos \beta + \sin
\beta)\\
&=& {1\over \sqrt{2}} \sqrt{p^2 + q^2}
\sin({\pi\over 4} + \beta).
\end{eqnarray*}
It follows that either $\beta = -{\pi \over 4}$,
${3\pi \over 4}$ or $p^2 + q^2 = 0$ which the latter occurs
when $\alpha = -{3\pi \over 4}$.

We first consider $\alpha = -{3\pi \over 4}$.  In
this situation, we have $p = 0$ and $q = 0$.  It
follows that $a_3 = b_3 = a_4 = b_4 = a_5 = b_5 =
0$.  It is clear that $p a_8 + qb_8 = -\sum_{\nu=a,b} \nu_5
\nu_7$ holds.  Thus, we only have two equations to
solve
$$
a_0 a_8 + b_0 b_8 = 0, \ \ \ a_8 + b_8 =-{1\over 2}(a_7 + b_7).
$$
Using $a_0 = {1\over 4} - a_1 + a_8$ and $b_0 =
{1\over 4} - b_1 + b_8$, we solve the above
equations and get
\begin{eqnarray*}
a_8 &=& {1\over 32}\bigg(-1 -\sqrt{2} \cos \gamma \pm
\sqrt{2 + \sqrt{2} (\cos \gamma + \sin \gamma)}\bigg)\\
b_8 &=& {1\over 32}\bigg(-1 -\sqrt{2} \sin \gamma \mp
\sqrt{2 + \sqrt{2}(\cos \gamma + \sin \gamma)}\bigg)\\
a_0 &=& {1\over 32}\bigg(1 +\sqrt{2} \cos \gamma \pm
\sqrt{2 + \sqrt{2} (\cos \gamma + \sin \gamma)}\bigg)\\
b_0 &=& {1\over 32}\bigg(1 +\sqrt{2} \sin \gamma \mp
\sqrt{2 + \sqrt{2} (\cos \gamma + \sin \gamma)}\bigg).
\end{eqnarray*}
Using the linear relationships, we have
\begin{eqnarray*}
a_2 &=& {1\over 32}\bigg(1 + \sqrt{2} \cos \gamma \mp
\sqrt{2 + \sqrt{2}(\cos \gamma + \sin \gamma)}\bigg)\\
b_2 &=& {1\over 32}\bigg(1 + \sqrt{2} \sin \gamma \pm
\sqrt{2 + \sqrt{2}(\cos \gamma + \sin \gamma)}\bigg)\\
a_6 &=& {1\over 32}\bigg(-1 - \sqrt{2} \cos \gamma \mp
\sqrt{2 + \sqrt{2}(\cos \gamma + \sin \gamma)}\bigg)\\
b_6 &=& {1\over 32}\bigg(-1 - \sqrt{2} \sin \gamma \pm
\sqrt{2 + \sqrt{2}(\cos \gamma + \sin \gamma)}\bigg)\\
a_1 &=& {3\over 16} - {\cos \gamma \over 8\sqrt{2}},
\ \ b_1 = {3\over 16} - {\sin \gamma \over
8\sqrt{2}}, \ \
a_7 = {1\over 16} + {\cos \gamma \over 8\sqrt{2}},
\ \ b_7 = {1\over 16} +{\sin \gamma \over
8\sqrt{2}}.
\end{eqnarray*}
Next we consider $\beta = -{\pi \over 4}$ or $\beta
= {3\pi \over 4}$.  We only need to consider $\beta = -{\pi
\over 4}$ while $p^2 + q^2 \ne 0$ because the other case
is a rotation of this one.  For $\beta = -{\pi\over 4}$,
we have three equations to solve:
\begin{equation}
a_0 a_8 + b_0 b_8 = 0, \ \ \
pa_8 + qb_8 = -\sum_{\nu=a,b} a_5 a_7, \ \ \
a_8 + b_8 = -{1\over 2} (a_7 + b_7).\label{(4.7)}
\end{equation}
Assuming $p \ne q$, i.e. $\alpha \ne {\pi \over
4}$, we can solve the second and third ones in the
above three equations for $a_8$ and $b_8$:  The
solutions for $a_8$ and $b_8$ are

\begin{eqnarray*}
b_8 &=& {1\over q - p}\bigg(-{p\over 2}(a_7 + b_7) + \sum_{\nu=a,b}
a_5 a_7\bigg)\\
a_8 &=& -{1\over 2}(a_7 + b_7) +{1\over p -
q}\bigg(-{p\over 2}(a_7 + b_7) + \sum_{\nu=a,b} a_5 a_7\bigg)
\end{eqnarray*}
leaving the relationship between $\alpha$ and
$\gamma$ as $a_0a_8 + b_0 b_8 = 0$.  Upon substitution and
simplification, we have
\begin{equation}
\sin(\alpha + {\pi \over 4}) - 2\sin(\gamma + {\pi
\over 4}) \sqrt{2(1 - \sin(\alpha + {\pi \over 4})}
+ \sin(2\gamma)-2 \ =\ 0. \label{(4.8)}
\end{equation}
This equation (\ref{(4.8)}) has a one parameter family of solutions given by
$$
\gamma=\left\{\begin{array}{ccccc}
-{1\over 2}\alpha+{7\pi \over 8}& {\rm or} & {1\over 2}\alpha-{3\pi \over 8}&
{\rm if} & 0< \alpha <{ \pi \over 4}\\
 {1\over 2}\alpha+{5\pi \over 8}& {\rm or} & -{1\over 2}\alpha-{ \pi \over 8}&
 {\rm if} & { \pi \over 4}< \alpha < 2\pi.
\end{array}\right.
$$

The final case is when $\beta=-{\pi\over 4}$ and $p=q$ (i.e.
$\alpha={\pi\over 4}$)
which implies that $$a_1=b_1=a_4=b_4={1\over 8},
a_3=a_7=b_5=b_7=0, a_5=b_3=-{1\over8}.$$
This reduces (\ref{(4.7)}) to
$$
a_0a_8+b_0b_8=0,\ \ \
a_8+b_8=0,
$$
which has two solutions: $a_8=b_8=0$ or $a_8={1\over 16}$ and
$b_8=-{1\over 16}$
yielding two rational solutions:
$$
{1\over 8}\left[\begin{array}{cccccc}
0 & 1 & 1 & 1 & 1 & 0\\
0 &0 &0 &0 &0& 0\\
0& -1 &1 &1& -1& 0\\
0& -1& 1& 1& -1& 0\\
0 &  0& 0& 0& 0 &0\\
0 & 1 & 1 & 1 & 1 & 0\end{array}\right],
\ \ \ \
{1\over 16}\left[\begin{array}{cccccc}
1 & 1 & 2 & 2 & 1 & 1\\
-1 &1 &0 &0 &1& -1\\
0& -2 &2 &2& -2& 0\\
0& -2& 2& 2& -2& 0\\
-1 & 1& 0& 0& 1 &-1\\
1 & 1 & 2 & 2 & 1 & 1\end{array}\right].
$$
\end{proof}

%Figure \ref{Figure 4.1} shows
%the second refinable function from solution 3) of
%Theorem \ref{Theorem 4.1} after 5 iterations of the cascade algorithm.
%\begin{center}
%\begin{figure}[ht]
%\vspace{24pt}
%\psone{plot2a.ps}{5}
%\caption{ An example refinable function from Subcase 2a}\label{Figure 4.1}
%\end{figure}
%\end{center}

%%%%%%%%%%%%%%%%
\subsection{Solution of Case 4}
%%%%%%%%%%%%%%%%
Finally, we consider Case 4 where both $a_1 +
a_4 + a_7 = 0$ and  $a_3 + a_4 + a_5 = 0$.  From equations (\ref{(2.16)})
and (\ref{(2.19)})
 we have $b_3 + b_4 +b_5 = 0$ and $b_1 + b_4 + b_7 = 0$.  As before,
using equations (\ref{(2.9)}), (\ref{(2.13)}) and (\ref{(2.14)}), we have
$$
\sum_{\nu=a,b}((\nu_0 + \nu_1 + \nu_2)^2 + (\nu_3 + \nu_4 + \nu_5)^2 +
(\nu_6 + \nu_7 + \nu_8)^2) = {1\over 8}.
$$
It follows that $\sum_{\nu=a,b}((\nu_0 + \nu_1 + \nu_2)^2 + (\nu_6 + \nu_7 +\nu_8)^2) =
{1\over 8}$.  Using equations (\ref{(2.15)}) and (\ref{(2.17)}), we have
$$
(a_0 + a_1 + a_2)^2 + (a_6 + a_7 + a_8)^2 ={1\over 16}.
$$
 By Lemma \ref{Lemma 2.2}, the above equation is
$$
({1\over 4} - (a_6 + a_7 + a_8))^2 + (a_6 +a_7 + a_8)^2 = {1\over 16}.
$$
Thus, $a_6 + a_7 + a_8 = {1\over 4}$ or $a_6 + a_7
+ a_8 = 0$.

Also, using equations (\ref{(2.8)}), (\ref{(2.12)}) and (\ref{(2.14)}), we have
$$
\sum_{\nu=a,b}\left((\nu_0 + \nu_3 + \nu_6)^2 + (\nu_1 + \nu_4 + \nu_7)^2 +
(\nu_2 + \nu_5 + \nu_8)^2 \right)= {1\over 8}.
$$
It follows that $\sum_{\nu=a,b}\left((\nu_0 + \nu_3 + \nu_6)^2 + (\nu_2 + \nu_5 + \nu_8)^2
\right)= {1\over 8}.$
So, we have
$$
(a_0 + a_3 + a_6)^2 + (a_2 + a_5 + a_8)^2 = {1\over 16}.
$$
Again by Lemma \ref{Lemma 2.2},
$\left({1\over 4} - (a_0 + a_3 + a_6)\right)^2 + (a_0 + a_3 +
a_6)^2 = {1\over 16}.$
Thus, $a_0 + a_3 + a_6 = 0$ or $a_0 + a_3 + a_6 = {1\over 4}$.

Therefore, we have four subcases to consider.  In addition to
$$a_1 + a_4 + a_7 =0, a_3 + a_4 + a_5 = 0,
b_1 + b_4 + b_7 =0, b_3 + b_4 + b_5 = 0, $$
we have
\begin{itemize}
\item Subcase 4a:
\begin{eqnarray*}
&&a_0 + a_1 + a_2 = {1\over 4},\ \ \ \  a_0 + a_3 + a_6 = {1\over 4}, \ \ \ \
a_6 + a_7 + a_8 = 0, \ \ \ \ a_2 + a_5 + a_8 = 0\\
&&b_0 + b_1 + b_2 ={1\over 4},\ \ \ \  b_0 + b_3 + b_6 = 0,\ \ \ \
b_6 + b_7 + b_8 =0, \ \ \ \ b_2 + b_5 + b_8 = {1\over 4}
\end{eqnarray*}
\item Subcase 4b:
\begin{eqnarray*}
&&a_0 + a_1 + a_2 = 0,  a_0 + a_3 + a_6 = {1\over 4},\ \ \ \
a_6 + a_7 + a_8 = {1\over 4},\ \ \ \  a_2 + a_5 + a_8 = 0,\\
&&b_0 + b_1 + b_2 =0,\ \ \ \   b_0 + b_3 + b_6 =0 ,\ \ \ \
b_6 + b_7 + b_8 ={1\over 4},\ \ \ \  b_2 + b_5 + b_8 = {1\over 4}\\
\end{eqnarray*}
\item Subcase 4c:
\begin{eqnarray*}
&&a_0 + a_1 + a_2 ={1\over 4} ,\ \ \ \   a_0 + a_3 + a_6 = 0,\ \ \ \
a_6 + a_7 + a_8 = 0,\ \ \ \  a_2 + a_5 + a_8 = {1\over 4}\\
&&b_0 + b_1 + b_2 ={1\over 4},\ \ \ \   b_0 + b_3 + b_6 = {1\over 4},\ \ \ \
b_6 + b_7 + b_8 =0,\ \ \ \  b_2 + b_5 + b_8 = 0,
\end{eqnarray*}
\item Subcase 4d:
\begin{eqnarray*}
&&a_0 + a_1 + a_2 =0 ,\ \ \ \   a_0 + a_3 + a_6 = 0,\ \ \ \
a_6 + a_7 + a_8 = {1\over 4},\ \ \ \  a_2 + a_5 + a_8 = {1\over 4}\\
&&b_0 + b_1 + b_2 =0 ,\ \ \ \   b_0 + b_3 + b_6 ={1\over 4},\ \ \ \
b_6 + b_7 + b_8 ={1\over 4},\ \ \ \  b_2 + b_5 + b_8 = 0.
\end{eqnarray*}
\end{itemize}
We only study the Subcase 4a and leave the
other three subcases to the interested reader.
With the linear constraints, we tackle the nonlinear conditions
(\ref{(2.2)})-(\ref{(2.14)}).  We use (\ref{(2.2)}) and (\ref{(2.3)})
to simplify (\ref{(2.8)}) and (\ref{(2.9)}) as follows:
\begin{eqnarray}
0 &=&
 \sum_{\nu=a,b} \nu_0 \nu_6 + \nu_1 \nu_7 + \nu_2 \nu_8 \ =\
\sum_{\nu=a,b}(\nu_0 + \nu_2)(\nu_6 + \nu_8) + \nu_1 \nu_7\nonumber\\
&=& \sum_{\nu=a,b}\left({1\over 4}  - \nu_1\right)(-\nu_7) + \nu_1 \nu_7
= 2\sum_{\nu=a,b}(\nu_1 \nu_7 - {1\over 8} \nu_7), \label{(5.1)}
\end{eqnarray}
and similarly,
\begin{eqnarray}
0 &=& \sum_{\nu=a,b} (\nu_0 \nu_2 + \nu_3 \nu_5 + \nu_6 \nu_8)
\ =\  \sum_{\nu=a,b}(\nu_0 + \nu_6)(\nu_2 + \nu_8) + \nu_3 \nu_5\nonumber\\
&=& 2\left(({1\over 4} - a_3)(-a_5) + a_3 a_5 +
(-b_3)({1\over 4} - b_5) + b_3b_5\right)\\
&=&  4\left(-{a_5\over 8} - {b_3\over 8} + a_3 a_5 + b_3 b_5\right).
\label{(5.2)}
\end{eqnarray}
As we did before, we have $a_4 = -{1\over 16} + {1\over 8\sqrt{2}} \cos \alpha$
and $b_4 = -{1\over 16} + {1\over 8\sqrt{2}} \sin \alpha$.  It follows that
\begin{equation}
a^2_4 + b^2_4 + {1\over 8}(a_4 + b_4) = 0.  \label{(5.3)}
\end{equation}
This fact will be used later.

With (\ref{(2.8)}), i.e., (\ref{(5.1)}), we can see that (\ref{(2.12)})
holds.  Indeed, the
left-hand side of (\ref{(2.12)}) is, by using (\ref{(5.3)}),
\begin{eqnarray*}
\sum_{\nu=a,b} &&\hspace{-.2in}(\nu_1(\nu_0 + \nu_2) + \nu_4(\nu_3 + \nu_5)
+ \nu_7(\nu_6 + \nu_8))
\ =\  \sum_{\nu=a,b}(\nu_1({1\over 4} - \nu_1) - a^2_4 - a^2_7)\\
&=& \sum_{\nu=a,b}({1\over 4}(-\nu_4 - \nu_7) - a^2_1 - a^2_7 - a^2_4)
\ =\  -\sum_{\nu=a,b}({1\over 4} \nu_4 + (\nu_1 + \nu_7)^2 + a^2_4)\\
&=& -\sum_{\nu=a,b}({1\over 4} \nu_4 + 2\nu_4^2)\  =\  0.
\end{eqnarray*}
Similarly, we can show that with (\ref{(2.9)}), i.e. (\ref{(5.2)}),
equation (\ref{(2.13)}) holds.

If we add equations (\ref{(2.10)}) and (\ref{(2.11)}) together,
we have by (\ref{(5.3)})
\begin{eqnarray*}
\sum_{\nu=a,b} &&\hspace{-.2in}\nu_4 (\nu_0 + \nu_2) + \nu_1(\nu_3 + \nu_5) +
\nu_7(\nu_3 + \nu_5) + \nu_4(\nu_6 + \nu_8)\\
&=&\sum_{\nu=a,b}(\nu_4({1\over 4} -\nu_1) - \nu_1\nu_4 - \nu_7 \nu_4 -
\nu_4 \nu_7)\\
&=& \sum_{\nu=a,b} {1\over 4} \nu_4 - 2\nu_4(\nu_1+ \nu_7)
\ =\  \sum_{\nu=a,b}{1\over 4} \nu_4 + 2\nu^2_4 \ =\  0.
\end{eqnarray*}
That is, the addition of (\ref{(2.10)}) and  (\ref{(2.11)}) is always true.
We only need to consider one of these
two equations.  Next we simplify equations (\ref{(2.4)}) - (\ref{(2.7)}).

Adding (\ref{(2.4)}) and (\ref{(2.8)}) with (\ref{(2.3)}), we have
\begin{eqnarray}
0 &=& \sum_{\nu=a,b}(\nu_2(\nu_7 + \nu_8) + (\nu_0 + \nu_1)
\nu_6 + \nu_1 \nu_7)\nonumber \\
&=& \sum_{\nu=a,b}({1\over 4}\nu_6 + \nu_1 \nu_7)
\ =\  \sum_{\nu=a,b}\left({1\over 4}\nu_6 + {\nu_7\over 8}
\right). \label{(5.4)}
\end{eqnarray}
Adding (\ref{(2.5)}) and (\ref{(2.8)}) together, we have
\begin{eqnarray}
0 &=& \sum_{\nu=a,b}(\nu_1 + \nu_2) \nu_8 + \nu_0(\nu_7 + \nu_6) +
\nu_1 \nu_7\nonumber\\
&=& \sum_{\nu=a,b}({1\over 4} \nu_8 + \nu_1 \nu_7) \ =\
\sum_{\nu=a,b} \left({1\over 4} \nu_8 +{1\over 8} \nu_8\right). \label{(5.5)}
\end{eqnarray}
It is easy to see that equations (\ref{(5.4)}) and (\ref{(5.5)})
are equivalent by using
$a_6+a_7+a_8=0$ and $b_0+b_1+b_2=0$.

Adding (\ref{(2.6)}) and (\ref{(2.9)}) together, we have
\begin{eqnarray}
0 &=& \sum_{\nu=a,b}(\nu_6(\nu_5 + \nu_8) +
\nu_2(\nu_0 + \nu_3) + \nu_3 \nu_5)\nonumber\\
&=& {1\over 4} a_2 + a_3 a_5 + {1\over 4} b_6 + b_3b_5
\ =\  {1\over 4}(a_2+b_6)+{1\over 8}(a_5+b_3). \label{(5.6)}
\end{eqnarray}
Adding (\ref{(2.7)}) and (\ref{(2.9)}) together with (\ref{(2.2)}), we have
\begin{eqnarray}
0 &=& \sum_{\nu=a,b}(\nu_8(\nu_3 + \nu_6) +
\nu_0(\nu_2 + \nu_5) + \nu_3 \nu_5)\nonumber\\
&=& {1\over 4} a_8 + a_3 a_5 + b_3 b_5 + {1\over 4} b_0
\ =\  {1\over 4}(a_8+b_0)+{1\over 8}(a_5+b_3). \label{(5.7)}
\end{eqnarray}
In fact, equations (\ref{(5.6)}) and (\ref{(5.7)}) are equivalent by using
$a_2+a_5+a_8=0$ and $b_0+b_3+b_6=0$.

Next we solve (\ref{(5.1)}) using $a_1 + a_4 + a_7 = 0$ and
$b_1 + b_4 + b_7 = 0$.
Letting $p = -a_4$ and $q = -b_4$, we have $pa_7 - a^2_7 -
{a_7 \over 8} + qb_7 -b^2_7 - {b_7 \over 8} = 0$.  That is
$$
(a_7 - {p\over 2} + {1\over 16})^2 + (b_7 - {q\over 2} + {1\over 16})^2 =
{1\over 4} \left((p - {1\over 8})^2 + (q - {1\over 8})^2\right).
$$
Recalling (\ref{(5.3)}), i.e., $p^2 - {1\over 4} p + q^2 - {1\over 4}q = 0$,
we have
\begin{eqnarray*}
(p - {1\over 8})^2 + (q - {1\over 8})^2 &=& {1\over 8}({1\over 4} - p - q) =
{1\over 8}\left({1\over 8} +
{1\over 8\sqrt{2}}(\cos \alpha + \sin \alpha)\right)\\
&=& {1\over 64}\left(1 + {1\over \sqrt{2}}(\cos \alpha + \sin \alpha)\right)
 \ =\  {1\over 64}\left(1 + \sin(\alpha + {\pi \over 4})\right).
\end{eqnarray*}
Thus,  we have
$$a_7 = {p\over 2}  - {1\over 16} + {1\over 16} \sqrt{1 + \sin(\alpha +
{\pi \over 4})} \cos \beta, \ \ \
b_7 = {q\over 2}  - {1\over 16} + {1\over 16} \sqrt{1 + \sin(\alpha +
{\pi \over 4})} \sin \beta.
$$
Similarly we can solve (\ref{(5.2)}) using $a_3 + a_4 + a_5 = 0$
and $b_3 + b_4 + b_5 =0 $ giving us
$$
a_5 = {p\over 2}  - {1\over 16} + {1\over 16} \sqrt{1 + \sin(\alpha +
{\pi \over 4})} \cos \gamma, \ \ \
b_3 = {q\over 2}  - {1\over 16} + {1\over 16} \sqrt{1 + \sin(\alpha +
{\pi \over 4})} \sin \gamma.
$$
It follows that
\begin{eqnarray*}
a_1 &= {p\over 2}  - {1\over 16} + {1\over 16} \sqrt{1 + \sin(\alpha +
{\pi \over 4})} \cos \beta, \ \ \
b_1 = {q\over 2}  - {1\over 16} + {1\over 16} \sqrt{1 + \sin(\alpha +
{\pi \over 4})} \sin \beta\\
a_3 &= {p\over 2}  +  {1\over 16} + {1\over 16} \sqrt{1 + \sin(\alpha +
{\pi \over 4})} \cos \gamma, \ \ \
b_5 = {q\over 2}  +  {1\over 16} + {1\over 16} \sqrt{1 + \sin(\alpha +
{\pi \over 4})} \sin \gamma.
\end{eqnarray*}
So, we still need to satisfy (\ref{(2.2)}), (\ref{(2.3)}), (\ref{(5.5)}),
(\ref{(5.7)}), and (\ref{(2.11)}).
Using the linear relationships for Subcase
4a, (\ref{(2.2)}) and (\ref{(2.3)})
become
\begin{eqnarray}
a_8^2+({1\over 4}-a_1+a_5)a_8+b_8^2(-b_1+b_5)b_8
&=&0\label{(5.8)}\\
a_8^2+(a_5+a_7)a_8+b_8^2+(b_5+b_7-{1\over 4})b_8
&=&-a_5a_7-b_5b_7+{b_7\over 4}.\label{(5.9)}
\end{eqnarray}
After subtracting (\ref{(5.8)}) from (\ref{(5.9)}) and using
$a_1+a_7=p$ and $b_1+b_7=q$,
we can replace one of these equations with
\begin{eqnarray}
(p-{1\over 4})a_8+(q-{1\over 4})b_8&={1\over 4}b_7-a_5a_7-b_5b_7.
\label{(5.10)}
\end{eqnarray}
Moreover, the linear relationships combined with (\ref{(2.11)}) yield
\begin{eqnarray*}
pa_8+qb_8&=&{1\over 8}a_4+{1\over 2}\sum_{\nu=a,b} \nu_1\nu_5+
\nu_3\nu_7+(\nu_5-\nu_1)\nu_4.
\end{eqnarray*}
Thus, (\ref{(5.5)}), (\ref{(5.7)}), and (\ref{(2.11)}) can be replaced by
\begin{eqnarray}
a_8+b_8 &=&-{1\over 2}(a_7+b_7)\label{(5.11)}\\
a_8+b_8 &=&-{1\over 2}(a_5-2b_1+b_3+2b_5)\label{(5.12)}\\
a_8+b_8&=&-b_7+{1\over 2}a_4+2\sum_{\nu=a,b}
2\nu_5\nu_7+\nu_1\nu_5+\nu_3\nu_7+(\nu_5-\nu_1)\nu_4.\label{(5.13)}
\end{eqnarray}
Now, we only need to satisfy (\ref{(5.8)}), (\ref{(5.10)}), and
(\ref{(5.11)})-(\ref{(5.13)}).

Equating the right-hand sides of (\ref{(5.11)}) and (\ref{(5.12)}) gives
$$
\sqrt{1 + \sin(\alpha + {\pi \over 4})}
(\cos \gamma-\sin \gamma-\cos \beta + \sin\beta)=0.
$$
This constraint is satisfied whenever $\alpha=-{3\pi \over 4},\gamma=\beta,$ or
$\gamma=-\beta-{\pi \over 2}$.

Equating the right hand sides of (\ref{(5.11)}) and (\ref{(5.13)}) gives
\begin{eqnarray*}
-{1\over 2}(a_7+b_7)&=&-b_7+{1\over 2}a_4+
2\sum_{\nu=a,b}(\nu_5(\nu_1+\nu_7)+(\nu_3+\nu_5)\nu_7+\nu_4\nu_5-\nu_1\nu_4\\
&=&-b_7+{1\over 2}a_4-2\sum_{\nu=a,b} \nu_4(\nu_1+\nu_7).
\end{eqnarray*}
Thus,
\begin{eqnarray*}
0&=&{1\over 2}a_4+2\sum_{\nu=a,b} \nu_4^2+{\nu_4\over 2}+{\nu_7\over 2}\\
&=&2(p^2+q^2)-{p\over 4}-{q\over 4}+
{1\over 32}\sqrt{1+\sin(\alpha+{\pi\over 4})}(\cos\beta-\sin\beta)\\
&=&{1\over 32}\sqrt{1+\sin(\alpha+{\pi\over 4})}(\cos\beta-\sin\beta).
\end{eqnarray*}
Therefore, we only need to solve (\ref{(5.8)}), (\ref{(5.10)}), and
(\ref{(5.11)}) when
$\alpha=-{3\pi \over 4}, \beta={\pi\over 4},$ or $\beta=-{3\pi \over 4}$
and $\gamma=\beta$ or $\gamma=-\beta-{\pi \over 2}$.

We begin with $\alpha=-{3\pi \over 4}$, then
$$
a_4=b_4=-{1 \over 8}, \ \ \ a_5=a_7=b_3=b_7=0,\ \ \
a_1=a_3=b_1=b_5={1 \over 8}
$$
which reduces (\ref{(5.8)}), (\ref{(5.10)}), and (\ref{(5.11)}) to
$$
(a_8+{1\over 8})a_8+b_8^2=0, \ \ \ a_8+b_8=0.
$$
These equations yield two solutions: $a_8=b_8=0$ or
$a_8=-{1\over 16}$ and $b_8={1\over 16}.$

Now, with $\beta=\gamma={\pi\over 4}$, we solve
for $a_8$ and $b_8$ using the linear equations (\ref{(5.10)}) and
(\ref{(5.11)}), we have
$$
a_8=-{p\over 4}-{\sqrt{2}\over 8}\sqrt{1+\sin(\alpha+\pi/4)}), \ \ \
b_8=-{q\over 4}+{1\over 16}.
$$
Plugging these solutions into (\ref{(5.8)}) and simplifying we have
$$
\cos(\alpha-{\pi\over 4})-1=4\sqrt{2+2\cos(\alpha-{\pi\over 4})}
$$
which has no real solution.

Similarly for $\beta=\gamma=-{3\pi\over 4}$, (\ref{(5.10)}) and
(\ref{(5.11)}) yield
$$
a_8=-{p\over 4}+{\sqrt{2}\over 8}\sqrt{1+\sin(\alpha+\pi/4)}), \ \ \
b_8=-{q\over 4}+{1\over 16},
$$
but now (\ref{(5.8)}) reduces to
$$
\cos(\alpha-{\pi\over 4})-1=-4\sqrt{2+2\cos(\alpha-{\pi\over 4})}
$$
which has two solutions: $\alpha={\pi\over 4}\pm\cos^{-1}(17-8\sqrt{5})$.

Finally for $\beta={\pi\over 4}$ and $\gamma=-{3\pi\over 4}$,
the linear equations
produce
$$
a_8=-{p\over 4}, \ \ \
b_8=-{q\over 4}+{1\over 16}-
{\sqrt{2}\over 32}\sqrt{1+\sin(\alpha-{\pi \over 4})},
$$
and similarly for $\beta=-{3\pi\over 4}$ and  $\gamma={\pi\over 4}$
$$
a_8=-{p\over 4}, \ \ \
b_8=-{q\over 4}+{1\over 16}+
{\sqrt{2}\over 32}\sqrt{1+\sin(\alpha-{\pi \over 4})}.
$$
Both of these choices for $a_8$ and $b_8$ after being substituted into
(\ref{(5.8)}) require
$$
p^2+q^2=0
$$
which is satisfied by $\alpha={\pi \over 4}$.

Therefore, the complete solution for Subcase 4a has 6 solitary
solutions.  The four rational solutions are given below:
\begin{eqnarray*}
&&{1\over 8}\left[\begin{array}{cccccc}
1&0&\ \ 1& \ \ 1 & 0 & 1\\
0&0&\ \ 0& \ \ 0 &0& 0\\
1&0&   -1&-1& 0& 1\\
1&0&   -1& -1& 0& 1\\
0&0&\ \ 0& \ \ 0& 0 &0\\
1&0&\ \ 1& \ \ 1 & 0 & 1\end{array}\right], \ \ \ \
{1\over 16}\left[\begin{array}{cccccc}
1 & \ \ 1 &  \ \ 2 &  \ \ 2 &  \ \ 1 & 1\\
1 &-1 & \ \ 0 &  \ \ 0 &-1& 1\\
2& \ \ 0 &-2 &-2&  \ \ 0& 2\\
2&  \ \ 0& -2& -2&  \ \ 0& 2\\
1 &  -1&  \ \ 0&  \ \ 0& -1 &1\\
1 &  \ \ 1 &  \ \ 2 &  \ \ 2 &  \ \ 1 & 1\end{array}\right],\\
&&{1\over 8}\left[\begin{array}{cccccc}
1 & 0 &  \ \ 1 &  \ \ 1 & 0 & 1\\
1 &0 &-1 &-1 &0& 1\\
0& 0 & \ \ 0 & \ \ 0& 0& 0\\
0& 0&  \ \ 0& \ \ 0& 0& 0\\
1 &  0& -1& -1& 0 &1\\
1 & 0 &  \ \ 1 &  \ \ 1 & 0 & 1\end{array}\right], \ \ \ \
{1\over 8}\left[\begin{array}{cccccc}
1 &  \ \ 1 & 0 & 0 &  \ \ 1 & 1\\
0 & \ \ 0 &0 &0 & \ \ 0& 0\\
1& -1 &0 &0& -1& 1\\
1& -1& 0& 0& -1& 1\\
0 &   \ \ 0& 0& 0&  \ \ 0 &0\\
1 &  \ \ 1 & 0 & 0 &  \ \ 1 & 1\end{array}\right].
\end{eqnarray*}

%Figure\ref{Figure 5.1} shows the second refinable function from above after 5 iterations of the cascade algorithm.

%\begin{center}
%\begin{figure}[ht]
%\vspace{24pt}
%\psone{plot4a.ps}{5}
%\caption{ An example refinable function from Subcase 4a}\label{Figure 5.1}
%\end{figure}
%\end{center}

%%%%%%%%%%%%%%%%%%%
\section{ORTHOGONALITY}
%%%%%%%%%%%%%%%%%%%
In this section, we discuss the orthogonality of the
solutions from Cases 1-4.  We begin with a
review of the Lawton condition as well as a well known
necessary and sufficient condition for orthogonality.
We conclude this section with a numerical experiment
consisting of a one-level decomposition of two gray-scale
images using a nonseparable filter from Subcase 4a,
and compare its performance with Haar and D4.

Let
$$m(e^{i\omega_1},e^{i\omega_2})=\sum_{k,\ell}
h_{k,l}e^{ik\omega_1}e^{i\ell\omega_2},
$$
where the $h_{k,\ell}$'s are the $a_i$'s and $b_i$'s as discussed
previously in Section \ref{Section 2}. Define
\begin{eqnarray}
\hat{\phi}(\omega_1,\omega_2)&=&
\prod_{k=1}^\infty m\left(\displaystyle e^{{i\omega_1} \over 2^k},
\displaystyle e^{{i\omega_2} \over 2^k}\right).
\label{(6.1)}
\end{eqnarray}
For the coefficients as defined in Section \ref{Section 2}, we know that
$\phi$ is well defined and is in $L_2({\bf R}^2)$.  Define
\begin{equation}
\alpha_{\ell_1,\ell_2}\ =\ \int_{{\bf R}^2}
\phi(x,y)\overline{\phi(x-\ell_1,y-\ell_2)}{\rm d}x{\rm d}y.\label{(6.2)}
\end{equation}
Thus, if $\alpha_{\ell_1,\ell_2}=\delta_{\ell_1,\ell_2}$, then $\phi$ is
orthonormal.  By the refinement equation
\begin{equation}
\phi(x,y)=4\sum_{k_1,k_2=0}^5 h_{k_1,k_2}\phi(2x-k_1,2y-k_2). \label{(6.3)}
\end{equation}
Using (\ref{(6.3)}) in (\ref{(6.2)}) we have
\begin{eqnarray}
\alpha_{\ell_1,\ell_2}
&=&4\sum_{n_1,n_2}\left(\sum_{k_1,k_2}h_{k_1,k_2}h_{k_1+n_1-2\ell_1,
k_2+n_2-2\ell_2} \right)\alpha_{n_1,n_2}.\label{(6.4)}
\end{eqnarray}
Because the $supp(\phi)\subset [0,5]^2$, the only possible
nonzero $\alpha_{\ell_1,\ell_2}$ are for $-4\leq\ell_1,\ell_2\leq 4$.
Let $\alpha$ be
the vector of length 81 consisting of the $\alpha_{\ell_1,\ell_2}$'s for some
fixed ordering of the indices in the range of $-4\leq\ell_1,\ell_2\leq 4$
and define the matrix
\begin{equation}
A_{(\ell_1,\ell_2),(n_1,n_2)}=4\sum_{k_1,k_2}h_{k_1,k_2}h_{k_1+n_1-2\ell_1,
k_2+n_2-2\ell_2}\label{(6.5)}
\end{equation}
for this same ordering.  Then equation (\ref{(6.4)}) says that $\alpha$ is an
eigenvector of $A$ with eigenvalue$\lambda=1$, i.e., $\alpha=A\alpha.$
Now, condition (i) of Section \ref{Section 1} implies that
$$
4\sum_{k_1,k_2}h_{k_1,k_2}h_{k_1-2j_1,k_2-2j_2}x^{2j_1}y^{2j_2}=1,
$$
i.e.
$$
4\sum_{k_1,k_2}h_{k_1,k_2}h_{k_1-2j_1,k_2-2j_2}=\delta_{j_1,j_2}.
$$
Thus the vector $\delta$ of length 81 consisting of the entries
$\delta_{\ell_1,\ell_2}$ for the same ordering as before is also
an eigenvector for $A$ with eigenvalue
$\lambda=1$.  For completeness, we state the generalization of
Lawton's  condition (cf. \cite{L}) in $R^2$.

\begin{theorem} \label{Theorem 6.1}  Let $m(x,y)$ be a given polynomial
satisfying (i) and (ii) and A a matrix defined as in equation (\ref{(6.5)})
for the coefficients of $m(x,y)$.
Let $\phi$ be the function generated by equation (\ref{(6.1)}).
If $\lambda=1$ is a non-degenerate eigenvalue of $A$, then
$\{\phi(\cdot-\ell_1,\cdot-\ell_2)|(\ell_1,\ell_2)\in{\bf Z}^2\}$
is an orthonormal set.
\end{theorem}

We also need the following well-known necessary and
sufficient condition for orthonormality.
\begin{theorem}\label{Theorem 6.2}  Let $m(x,y)$ be a given polynomial
satisfying (i) and (ii). Let $\phi$ be the function generated by
equation (\ref{(6.1)}). Then
$\{\phi(\cdot-\ell_1,\cdot-\ell_2)|(\ell_1,\ell_2)\in{\bf Z}^2\}$
is an orthonormal set if and only if
$$ \sum_{(k,\ell)\in {\bf Z}^2} |\widehat{\phi}((\omega_1,\omega_2)+
2\pi (k,\ell))|^2 = 1, \quad \forall (\omega_1,\omega_2)\in
[-\pi,\pi]^2.$$.
\end{theorem}

{\bf Case 1:}  We have used Matlab and Mathematica to
check the eigenvalues of the Lawton matrix associated with
this two-parameter family for a large sample of parameters.
The eigenvalue $\lambda=1$ was non-degenerate for every sample
we tested.

{\bf Case 2:}  These solutions are not associated with scaling
functions.
The conditions for Case 2a immediately imply
$\displaystyle
m(e^{i\omega_1},1)=(1+e^{5i\omega_1})/2.$
If we consider the one-dimensional
restriction $\bar{m}(\omega):=m(e^{i\omega},1)$, we see that
$\{\pm{\pi \over 5}, \pm{3\pi \over 5},\pm\pi\}$ are the zeros
of $\bar{m}(\omega)$.  Because $m(e^{i\omega},-1)=0$,
condition (ii) implies that $|\bar{m}(\omega)|^2+
|\bar{m}(\omega+\pi)|^2=1$.  Moreover,
$$
\left|\bar{m}\left(-{3\pi \over 5}+\pi\right)\right|=
\left|\bar{m}\left(-{\pi \over 5}+\pi\right)\right|=
\left|\bar{m}\left({3\pi \over 5}-\pi\right)\right|=
\left|\bar{m}\left({\pi \over 5}-\pi\right)\right|=1.
$$
Because $\left\{\xi_1={2\pi \over 5},\xi_2={4\pi \over 5},\xi_3
=-{2\pi \over 5},\xi_4=-{4\pi \over 5} \right\}$
is a nontrivial cycle in $[-\pi,\pi]$ for the operation
$\xi\rightarrow 2\xi$ mod $2\pi$ such that
$|\bar{m}(\xi_i)|=1$, the set of functions
$\{\bar{\phi}(\cdot-n)\}_{n\in{\bf Z}}$ associated with $\bar{m}(\omega)$ is
not orthonormal and
$$\sum_k |\hat{\bar{\phi}}({2\pi \over 5}+2\pi k)|^2=0$$
(See \cite{D}).  So, for the unrestricted function, we have
$$\sum_{k,l} \left|\hat{\phi}\left({2\pi \over 5}+2\pi k, 2\pi\ell\right)
\right|^2=0$$
which contradicts Theorem 6.2.

The other subcase 2b has the property that
$\displaystyle m(e^{i\omega_1},1)=(1+e^{i3\omega_1})/2$
which similarly excludes it from being associated with scaling functions.

{\bf Case 3:}  This case has the same problematic factors
as Case 2 but with respect to the other component
$\displaystyle m(1,e^{i\omega_2})=(1+e^{i5\omega_2})/2$ or
$\displaystyle m(1,e^{i\omega_2})=(1+e^{i3\omega_2})/2$.

{\bf Case 4:}  The solutions for this final case have the same
factors as in Cases 2 and 3.

So, Case 1 is the only solution associated with scaling functions
since the other cases failed to satisfy the necessary and sufficient
condition for othogonality. Although, the refinable functions for
Cases 2-4 are not orthogonal to their shifts they still have
associated tight frames since they satisfy condition (ii).  These
cases are analogous to the univariate Haar function with support
$[0,3]$.

%%%%%%%%%%%%%%%%%%%%%
\section{The $8\times 8$ Case}
%%%%%%%%%%%%%%%%%%%%

In this section,  We derive
several necessary conditions from the properties
(i)-(iv) with $N=7$ and $M=1$.  We will use these necessary
conditions to show in the next section that the second order
vanishing moment $M=2$
is not possible for this support size.

Let $N=7$ and consider
$m(x,y)=\displaystyle \sum_{i=0}^7\sum_{j=0}^7 c_{ij} x^i y^j$ which
satisfies properties (i)-(iv) with $M=1$.  The symmetry property (iii)
implies that
\begin{equation}
m(x,y)= \left[\begin{array}{c}
1\\ y\\ y^2\\ y^3\\ y^4\\ y^5\\ y^6\\ y^7
\end{array}\right]^T
\left[\begin{array}{cccccccc}
a_0 & b_0 & a_1 & b_1 & a_2 & b_2 & a_3& b_3\\
b_{15} & a_{15} & b_{14} & a_{14} & b_{13} & a_{13} & b_{12}& a_{12}\\
a_4 & b_4 & a_5 & b_5 & a_6 & b_6 & a_7& b_7\\
b_{11} & a_{11} & b_{10} & a_{10} & b_{9} & a_{9} & b_{8}& a_{8}\\
a_8 & b_8 & a_9 & b_9 & a_{10} & b_{10} & a_{11}& b_{11}\\
b_{7} & a_{7} & b_{6} & a_{6} & b_{5} & a_{5} & b_{4}& a_{4}\\
a_{12} & b_{12} & a_{13} & b_{13} & a_{14} & b_{14} & a_{15}& b_{15}\\
b_{3} & a_{3} & b_{2} & a_{2} & b_{1} & a_{1} & b_{0}& a_{0}
\end{array}\right]
\left[\begin{array}{c}
1\\ x\\ x^2\\ x^3\\ x^4\\ x^5\\ x^6\\ x^7
\end{array}\right].\label{eq1}
\end{equation}

Properties (i) and (iii)  imply that
\begin{equation}
\sum_{i=0}^{15} \sum_{\nu=a,b} \nu_i=\frac{1}{2}. \label{(8.1)}
\end{equation}

Property (ii) implies the following 25 nonlinear equations:
\begin{eqnarray}
&&\sum_{\nu=a,b}{} \nu_0\nu_{15} =0 \label{(8.2)}\\
&&\sum_{\nu=a,b}{} \nu_3 \nu_{12} =0 \label{(8.3)}\\
&&\sum_{\nu=a,b}{} \nu_0\nu_{11}+\nu_4\nu_{15}= 0 \label{(8.4)}\\
&&\sum_{\nu=a,b}{} \nu_0\nu_{14}+\nu_1\nu_{15}=0 \label{(8.5)}\\
&&\sum_{\nu=a,b}{} \nu_2\nu_{12}+\nu_3\nu_{13}=0 \label{(8.6)}\\
&&\sum_{\nu=a,b}{} \nu_3\nu_8+\nu_7\nu_{12}=0 \label{(8.7)}\\
&&\sum_{\nu=a,b}{} \nu_0\nu_7+\nu_4\nu_{11}+\nu_8\nu_{15}=0 \label{(8.8)}\\
&&\sum_{\nu=a,b}{} \nu_0\nu_{13}+\nu_1\nu_{14}+\nu_2\nu_{15} =0 \label{(8.9)}\\
&&\sum_{\nu=a,b}{} \nu_1\nu_{12}+\nu_2\nu_{13}+\nu_3\nu_{14}=0 \label{(8.10)}\\
&&\sum_{\nu=a,b}{} \nu_3\nu_4+\nu_7\nu_8+\nu_{11}\nu_{12}=0 \label{(8.11)}\\
&&\sum_{\nu=a,b}{} \nu_0\nu_3+\nu_4\nu_7+\nu_8\nu_{11}+\nu_{12}\nu_{15}=0 \label{(8.12)}\\
&&\sum_{\nu=a,b}{} \nu_0\nu_{10}+\nu_1\nu_{11}+\nu_4\nu_{14}+\nu_5\nu_{15}=0 \label{(8.13)}\\
&&\sum_{\nu=a,b}{} \nu_0\nu_{12}+\nu_1\nu_{13}+\nu_2\nu_{14}+\nu_3\nu_{15}=0 \label{(8.14)}\\
&&\sum_{\nu=a,b}{} \nu_2\nu_8+\nu_3\nu_9+\nu_6\nu_{12}+\nu_7\nu_{13}=0 \label{(8.15)}\\
&&\sum_{\nu=a,b}{} \nu_0\nu_6+\nu_1\nu_7+\nu_4\nu_{10}+\nu_5\nu_{11}+\nu_8\nu_{14}+\nu_9\nu_{15}=0 \label{(8.16)}\\
&&\sum_{\nu=a,b}{} \nu_0\nu_9+\nu_1\nu_{10}+\nu_2\nu_{11}+\nu_4\nu_{13}+\nu_5\nu_{14}+\nu_6\nu_{15}=0 \label{(8.17)}\\
&&\sum_{\nu=a,b}{} \nu_1\nu_8+\nu_2\nu_9+\nu_3\nu_{10}+\nu_5\nu_{12}+\nu_6\nu_{13}+\nu_7\nu_{14}=0 \label{(8.18)}\\
&&\sum_{\nu=a,b}{} \nu_2\nu_4+\nu_3\nu_5+\nu_6\nu_8+\nu_7\nu_9+\nu_{10}\nu_{12}+\nu_{11}\nu_{13}=0 \label{(8.19)}\\
&&\sum_{\nu=a,b}{} \nu_0\nu_8+\nu_1\nu_9+\nu_2\nu_{10}+\nu_3\nu_{11}+\nu_4\nu_{12}+\nu_5\nu_{13}+\nu_6\nu_{14}
\hspace{.25in}\nonumber\\
&&\hspace{3in}+\nu_7\nu_{15}=0 \label{(8.20)}\\ \nonumber\\
&&\sum_{\nu=a,b}{} \nu_0\nu_2+\nu_1\nu_3+\nu_4\nu_6+\nu_5\nu_7+\nu_8\nu_{10}+\nu_9\nu_{11}+\nu_{12}\nu_{14}
\hspace{.25in}\nonumber\\
&&\hspace{3in}+\nu_{13}\nu_{15}=0 \label{(8.21)}\\ \nonumber\\
&&\sum_{\nu=a,b}{} \nu_0\nu_5+\nu_1\nu_6+\nu_2\nu_7+\nu_4\nu_9+\nu_5\nu_{10}+
\nu_6\nu_{11}+\nu_8\nu_{13}+\nu_9\nu_{14}
\hspace{.25in}\nonumber\\
&&\hspace{3in}+\nu_{10}\nu_{15}=0 \label{(8.22)}\\ \nonumber\\
&&\sum_{\nu=a,b}{} \nu_1\nu_4+\nu_2\nu_5+\nu_3\nu_6+\nu_5\nu_8+\nu_6\nu_9+
\nu_7\nu_{10}+\nu_9\nu_{12}+\nu_{10}\nu_{13}
\hspace{.25in}\nonumber\\
&&\hspace{3in}+\nu_{11}\nu_{14} =0 \label{(8.23)}\\ \nonumber\\
&&\sum_{\nu=a,b}{} \nu_0\nu_1+\nu_1\nu_2+\nu_2\nu_3+\nu_4\nu_5+\nu_5\nu_6+
\nu_6\nu_7+\nu_8\nu_9+\nu_9\nu_{10}+\nu_{10}\nu_{11}
\hspace{.25in}\nonumber\\
&&\hspace{2in}+\nu_{12}\nu_{13}+\nu_{13}\nu_{14}+\nu_{14}\nu_{15} =0 \label{(8.24)}\\ \nonumber\\
&&\sum_{\nu=a,b}{} \nu_0\nu_4+\nu_1\nu_5+\nu_2\nu_6+\nu_3\nu_7+
\nu_4\nu_8+\nu_5\nu_9 +\nu_6\nu_{10}+\nu_7\nu_{11}+\nu_8\nu_{12}
\hspace{.25in}\nonumber\\
&&\hspace{2in}+\nu_9\nu_{13}+\nu_{10}\nu_{14}+\nu_{11}\nu_{15}=0  \label{(8.25)}\\ \nonumber\\
 \displaystyle &&\sum_{i=0}^{15}{} \sum_{\nu=a,b} \nu_i^2 =\frac{1}{8}.
\label{(8.26)}
\end{eqnarray}

Property (iv) with $M=1$ implies

\begin{eqnarray}
a_0+a_1+a_2+a_3&=&b_0+b_1+b_2+b_3 \label{(8.27)}\\
a_4+a_5+a_6+a_7&=&b_4+b_5+b_6+b_7 \label{(8.28)}\\
a_8+a_9+a_{10}+a_{11}&=&b_8+b_9+b_{10}+b_{11} \label{(8.29)}\\
a_{12}+a_{13}+a_{14}+a_{15} &=& b_{12}+b_{13}+b_{14}+b_{15} \label{(8.30)}\\
a_0+a_4+a_8+a_{12}&=&b_3+b_7+b_{11}+b_{15} \label{(8.31)} \\
a_1+a_5+a_9+a_{13}&=&b_2+b_6+b_{10}+b_{14} \label{(8.32)} \\
a_2+a_6+a_{10}+a_{14}&=&b_1+b_5+b_9+b_{13} \label{(8.33)} \\
a_3+a_7+a_{11}+a_{15}&=&b_0+b_4+b_8+b_{12}. \label{(8.34)}
\end{eqnarray}
Using (\ref{(8.27)})-(\ref{(8.30)}) and (\ref{(8.1)}), we immediately have
\begin{equation}
\sum_{i=0}^{15}a_i=\sum_{i=0}^{15} b_i=\frac{1}{4}. \label{(8.35)}
\end{equation}

Next, we use various combinations of the nonlinear equations to make
perfect squares as we have done previously.
This enables us to introduce parameters in order to
simplify these equations.

Using (\ref{(8.26)}), (\ref{(8.25)}), (\ref{(8.20)}),
and (\ref{(8.14)}),
we have
\begin{eqsplitmath}
\sum_{\nu=a,b} (\nu_0+\nu_4+\nu_8+\nu_{12})^2+
(\nu_1+\nu_5+\nu_9+\nu_{13})^2 \\
 +(\nu_2+\nu_6+\nu_{10}+\nu_{14})^2+
(\nu_3+\nu_7+\nu_{11}+\nu_{15})^2 ={1\over 8}.
\label{(9.5)}
\end{eqsplitmath}

Moreover, using (\ref{(8.5)}), (\ref{(8.6)}), (\ref{(8.13)}), (\ref{(8.15)}),
(\ref{(8.16)}), (\ref{(8.19)}), (\ref{(8.21)}), we have
\begin{eqsplitmath}
\sum_{\nu=a,b} (\nu_0+\nu_4+\nu_8+\nu_{12})(\nu_2+\nu_6+\nu_{10}+\nu_{14}) \\
+(\nu_1+\nu_5+\nu_9+\nu_{13})(\nu_3+\nu_7+\nu_{11}+\nu_{15}) =0.
\label{(9.6)}
\end{eqsplitmath}

After using (\ref{(8.31)})-(\ref{(8.34)}), (\ref{(9.5)}), and (\ref{(9.6)}),
we obtain
\begin{eqsplitmath}
((a_0+a_4+a_8+a_{12})\pm(a_2+a_6+a_{10}+a_{14}))^2 \\
+((a_1+a_5+a_9+a_{13})\pm(a_3+a_7+a_{11}+a_{15}))^2 =\frac{1}{8}.
\label{(9.6b)}
\end{eqsplitmath}
Choosing the plus sign in equation (\ref{(9.6b)}) and using (\ref{(8.35)})
yields
\begin{eqnarray}
a_0+a_4+a_8+a_{12}+a_2+a_6+a_{10}+a_{14}&=& r_0, \\
a_1+a_5+a_9+a_{13}+a_{3}+a_{7}+a_{11}+a_{15}&=&s_0,
\end{eqnarray}
where $r_0+s_0=1/4$ and $r_0s_0=0$.  Giving us four cases:
$r_0,s_0=1/4$ or $0$.

Similarly, using (\ref{(8.26)}), (\ref{(8.24)}), (\ref{(8.21)}),
and (\ref{(8.12)}), we have
\begin{eqsplitmath}
\sum_{\nu=a,b} (\nu_0+\nu_1+\nu_2+\nu_3)^2+(\nu_4+\nu_5+\nu_6+\nu_7)^2\\
+(\nu_8+\nu_9+\nu_{10}+\nu_{11})^2+(\nu_{12}+\nu_{13}+\nu_{14}+\nu_{15})^2 ={1\over 8}.
\label{(9.1)}
\end{eqsplitmath}
Using (\ref{(8.20)}), (\ref{(8.18)}), (\ref{(8.17)}), (\ref{(8.15)}),
(\ref{(8.13)}), (\ref{(8.7)}), (\ref{(8.4)}), we have
\begin{eqsplitmath}
\sum_{\nu=a,b} (\nu_0+\nu_1+\nu_2+\nu_3)(\nu_8+\nu_9+\nu_{10}+\nu_{11})\\
+(\nu_4+\nu_5+\nu_6+\nu_7)(\nu_{12}+\nu_{13}+\nu_{14}+\nu_{15})=0. \label{(9.2)}
\end{eqsplitmath}
In a similar fashion as before, we have
\begin{eqnarray}
a_0+a_1+a_2+a_3+ a_8+a_9+a_{10}+a_{11}&=&t_0,\\
a_4+a_5+a_6+a_7 +a_{12}+a_{13}+a_{14}+a_{15}&=&u_0,
\end{eqnarray}
where $t_0+u_0=1/4$ and $t_0u_0=0$.
We now refine our solutions to sums of four coefficients.
Choosing the minus sign in equation (\ref{(9.6b)}) yields
\begin{eqnarray*}
a_0+a_4+a_8+a_{12} &=&\frac{1}{2}(r_0+r_1),\ \
a_2+a_6+a_{10}+a_{14}\ =\ \frac{1}{2}(r_0-r_1) \\
a_1+a_5+a_9+a_{13} &=&\frac{1}{2}(s_0+s_1),\ \
a_{3}+a_{7}+a_{11}+a_{15}\ =\ \frac{1}{2}(s_0-s_1),
\end{eqnarray*}
where $r_1=\frac{1}{4}\cos\alpha$ and $s_1=\frac{1}{4}\sin\alpha$.
Furthermore, equations (\ref{(8.2)}), (\ref{(8.3)}), (\ref{(8.4)}),
(\ref{(8.7)}), (\ref{(8.8)}), (\ref{(8.11)}), and (\ref{(8.12)}) yield
$$
\sum (a_0+a_4+a_8+a_{12})(a_{3}+a_{7}+a_{11}+a_{15}) =0. \label{(9.7)}
$$
This together with (\ref{(8.31)}) and (\ref{(8.34)}), give the following
constraint on our parameters
\begin{equation}
2(r_0+r_1)(s_0-s_1)=0.\label{rseq}
\end{equation}
Because of the relationship between $r_1$ and $s_1$, equation (\ref{rseq}) produces
three cases:  $r_1=r_0$ and  $s_1=s_0$, $r_1=-r_0$ and  $s_1=s_0$, or
$r_1=-r_0$ and  $s_1=-s_0$.

Similarly,
\begin{eqnarray*}
a_0+a_1+a_2+a_3 &=&\frac{1}{2}(t_0+t_1),\ \
a_8+a_9+a_{10}+a_{11}\ =\ \frac{1}{2}(t_0-t_1) \\
a_4+a_5+a_6+a_7 &=&\frac{1}{2}(u_0+u_1), \ \
a_{12}+a_{13}+a_{14}+a_{15}\ =\ \frac{1}{2}(u_0-u_1),
\end{eqnarray*}
where
$t_1=\frac{1}{4}\cos \beta$ and $u_1=\frac{1}{4}\sin \beta$.
Additionally, equations (\ref{(8.2)}),  (\ref{(8.3)}),  (\ref{(8.5)}),
(\ref{(8.9)}), (\ref{(8.10)}),  and (\ref{(8.14)}) with
(\ref{(8.27)}) and (\ref{(8.30)}), gives us $(t_0+t_1)(u_0-u_1)=0$.
In summary, we have the following lemma
\begin{lemma}
\label{lems4}
$$
\begin{array}{ll}
a_0+a_4+a_8+a_{12} \ =\ \frac{1}{2}(r_0+r_1),&
b_0+b_4+b_8+b_{12}\ =\ \frac{1}{2}(s_0-s_1),\\
a_1+a_5+a_9+a_{13}\ =\ \frac{1}{2}(s_0+s_1),&
b_1+b_5+b_9+b_{13}\ =\ \frac{1}{2}(r_0-r_1)\\
a_2+a_6+a_{10}+a_{14}\ =\ \frac{1}{2}(r_0-r_1),&
b_2+b_6+b_{10}+b_{14}\ =\ \frac{1}{2}(s_0+s_1),\\
a_{3}+a_{7}+a_{11}+a_{15}\ =\ \frac{1}{2}(s_0-s_1),&
b_{3}+b_{7}+b_{11}+b_{15}\ =\ \frac{1}{2}(r_0+r_1),\\
a_0+a_1+a_2+a_3 \ =\  \frac{1}{2}(t_0+t_1),&
b_0+b_1+b_2+b_3 \ =\ \frac{1}{2}(t_0+t_1),\\
a_4+a_5+a_6+a_7 \ =\ \frac{1}{2}(u_0+u_1),&
b_4+b_5+b_6+b_7 \ =\ \frac{1}{2}(u_0+u_1),\\
a_8+a_9+a_{10}+a_{11}\ =\  \frac{1}{2}(t_0-t_1),&
b_8+b_9+b_{10}+b_{11}\ =\  \frac{1}{2}(t_0-t_1),\\
a_{12}+a_{13}+a_{14}+a_{15}\ = \  \frac{1}{2}(u_0-u_1),&
b_{12}+b_{13}+b_{14}+b_{15}\ =\ \frac{1}{2}(u_0-u_1).
\end{array}
$$
where
\begin{eqnarray*}
r_0+s_0=\frac{1}{4}, \ \ r_0s_0=0, \ \ (r_0+r_1)(s_0-s_1)=0,
\ \ r_1=\frac{1}{4}\cos \alpha, \ \ s_1\ =\ \frac{1}{4}\sin \alpha,\\
 t_0+u_0=\frac{1}{4} ,  \ \ t_0u_0=0 ,  \ \ (t_0+t_1)(u_0-u_1)=0 ,
\ \  t_1=\frac{1}{4}\cos \beta , \ \  u_1\ =\ \frac{1}{4}\sin \beta.
\end{eqnarray*}
\end{lemma}

Because there are no symmetric compactly-supported tensor-product
wavelets with more than one vanishing moment, it is natural to ask
whether there are any nonseparable symmetric solutions with multiple
vanishing moments.  In \cite{LR1}, it is shown that their are no
symmetric solutions with higher vanishing moments for $m(x,y)$ with
$N=7$.   Although the bivariate case allows enough freedom to
generate a family of symmetric solutions, it does not allow for
multiple vanishing moments at least for the support size we have
considered. So, compact support, orthogonality, vanishing moments,
and symmetry are again at odds in the construction of bivariate
wavelets.
%% End of article:

%% optional
% Appendixes

% Appendix without title:
%\appendix{}

% Appendix with title:
%\appendix{Title}

% Appendix with letter:
%\appendix{B}

% Appendix with letter and title:
%\appendix{C}
%\appendixtitle{This is an appendix title}

%% optional
%\begin{acknowledgment}
%text...
%\end{acknowledgment}

%% not optional:

%% This command is necessary! ==>>
\end{article}
\end{document}